\documentclass{amsart}
\usepackage{amsmath,amssymb,amsthm,enumerate,amscd}
\usepackage{url}
\usepackage{array}
\usepackage{tikz}
\usepackage{siunitx}

\newtheorem{theorem}{Theorem}[section]
\newtheorem*{theorem*}{Theorem}
\newtheorem{lemma}[theorem]{Lemma}

\newtheorem{proposition}[theorem]{Proposition}

\newtheorem{algorithm}[theorem]{Algorithm}
\newtheorem{definition}[theorem]{Definition}
\newtheorem{preremark}[theorem]{Remark}
	\newenvironment{remark}%
		{\begin{preremark}\rm}{\end{preremark}}
\newtheorem{preexample}[theorem]{Example}

\def\f{F}
\def\Z{\mathbb{Z}}
\def\Q{\mathbb{Q}}
\def\F{\mathbb{F}}

\def\R{\mathbb{R}}
\def\N{\mathbb{N}}

\def\m{\mathfrak{M}}
\def\MM{\textsf{M}}
\def\M{{M}}
\def\ch{\mathrm{char}}
\newcommand{\OK}{\mathcal{O}_K}
\newcommand{\OH}{\mathcal{O}_H}

\renewcommand{\O}{\mathcal{O}}

\newcommand{\p}{\mathfrak{p}}
\newcommand{\q}{\mathfrak{q}}
\def\a{\mathfrak{a}}
\def\b{\mathfrak{b}}
\newcommand{\norm}{\operatorname{N}}
\newcommand{\disc}{\operatorname{disc}}
\newcommand{\End}{\operatorname{End}}
\newcommand{\Aut}{\operatorname{Aut}}
\newcommand{\Gal}{\operatorname{Gal}}
\newcommand{\ord}{\operatorname{ord}}

\newcommand{\frakp}{\mathfrak{p}}

\title[Using CM elliptic curves for deterministic primality proving]
{A framework for deterministic primality proving using elliptic curves with complex multiplication}
\author[A. Abatzoglou]{Alexander Abatzoglou}
\address{Department of Mathematics, University of California, Irvine, CA 92697}
\email{aabatzog@math.uci.edu}
\author[A. Silverberg]{Alice Silverberg}
\address{Department of Mathematics, University of California, Irvine, CA 92697}
\email{asilverb@math.uci.edu}
\author[A.V. Sutherland]{\\Andrew V.\ Sutherland}
\address{Department of Mathematics, MIT, Cambridge, MA 02139}
\email{drew@math.mit.edu}
\author[A. Wong]{Angela Wong}
\address{Department of Mathematics, University of California, Irvine, CA 92697}
\email{awong@math.uci.edu}
\thanks{This work was supported by the National Science Foundation under grants CNS-0831004 and DMS-1115455.}

\begin{document}

\begin{abstract}
We provide a framework for using elliptic curves with complex multiplication to
determine the primality or compositeness of integers that lie in special sequences,
in deterministic quasi-quadratic time.  We use this to find large primes, including the largest prime currently known whose primality cannot feasibly be proved using classical methods.
\end{abstract}

\maketitle

\section{Introduction}
\label{intro}

The problem of determining whether a given integer is prime or composite is a fundamental problem of computational number theory, with a history that dates back to Gauss, and even Eratosthenes.
The most significant modern result is that of Agrawal, Kayal, and Saxena \cite{aks04}, who gave the first deterministic polynomial-time algorithm for determining primality.
As refined by Lenstra and Pomerance \cite{lenstrapomerance11}, the running time of this algorithm is $\tilde{O}(n^6)$, where $n$ is the number of bits in the binary representation of the given integer and the $\tilde{O}$ notation ignores factors that are polynomial in $\log n$.
This is the best general result currently known (for deterministic algorithms), but for numbers of certain special forms, such as Mersenne numbers and Fermat numbers, one can do better; it is possible to determine the primality or compositeness of these and other special types of numbers in $\tilde{O}(n^2)$ time.
It is these much faster algorithms that have made it possible to prove the primality of numbers like $2^{57,885,161}-1$, which is the largest prime currently known.
 
Pomerance has shown that every prime number $p$, not just those of a special form, admits a short proof of its primality, one that can be verified in $\tilde{O}(n^2)$ time~\cite{pomerance87}.
This proof is known as a \emph{Pomerance proof}, and it takes the form of a point on an elliptic curve.
We also note that every composite number $p$ admits a short proof of its compositeness; its factorization is one, and there are others that can be efficiently found using a probabilistic algorithm --- this is the basis of the Miller-Rabin test for compositeness \cite{miller76, rabin80}.
Thus in principle, every integer admits a short proof of its primality or compositeness, one that can be verified in $\tilde{O}(n^2)$ time, but at present we do not have an efficient way of finding such a proof, except in special cases.
Indeed, no polynomial-time algorithm for finding a Pomerance proof for a given prime is known, not even a probabilistic one.

In this paper we consider a framework for determining the primality of integers that lie in certain special sequences, using elliptic curves with complex multiplication (CM).
The running time of these algorithms is $\tilde{O}(n^2)$, and in the case that the given integer is prime, they produce a Pomerance proof of its primality.
We are particularly interested in cases where classical methods that do not use elliptic curves are inapplicable; such methods require   a partial factorization of either $N+1$ or $N-1$, where $N$ is the integer being tested for primality \cite[Ch. 4]{crandallpomerance}.

Goldwasser and Kilian gave a general primality proving algorithm using randomly generated elliptic curves in \cite{Goldwasser-Kilian86,goldwasserkilian99}. Atkin and Morain improved their algorithm by using elliptic curves constructed via the CM method, rather than random elliptic curves \cite{atkinmorain93}.  With further asymptotic improvements due to Shallit, the Atkin-Morain algorithm has a heuristic expected running time of $\tilde{O}(n^4)$  \cite{morain07}, and is currently the most widely used method for general purpose primality proving.
The time to verify the certificate produced by an elliptic curve
primality proof is $\tilde{O}(n^3)$. 
We also note a general purpose compositeness test of Gordon \cite{gordon1989} that uses supersingular reductions of CM elliptic curves over~$\Q$, and a proposal by Mihailescu \cite{Mihailescu07,FKDG} for a general purpose primality proving algorithm that uses a combination of elliptic curve and cyclotomy primality proving methods.

During the mid 1980s, Bosma in \cite{bosma85} and Chudnovsky and Chudnovsky in \cite{chudnovsky86} proposed primality testing algorithms for numbers of special forms 
using elliptic curves with CM.  Bosma proposed primality tests based on elliptic curves with CM by $\Q(i)$ and $\Q(\sqrt{-3})$.  Chudnovsky and Chudnovsky proposed a general method for proving the primality of numbers in certain sequences based on
algebraic varieties, including CM elliptic curves, but their approach is not guaranteed to succeed on every prime input (so it does not prove compositeness).

In \cite{gross04}, Gross used an elliptic curve with CM by $\Q(i)$ to obtain an efficient primality proving algorithm for Mersenne numbers.
Gross's algorithm runs in $\tilde{O}(n^2)$ time, but the same is true of the classical Lucas-Lehmer test \cite{pomerance10}, which does not use elliptic curves, and the constant factors hidden in the $\tilde{O}$-notation favor the latter.
Denomme and Savin \cite{denommesavin08} and Tsumura \cite{tsumura11}, obtained similar results for other sequences of numbers, but in both cases there are faster classical methods that can be applied to these sequences.   In \cite{Gurevich12}, Gurevich and Kunyavski\u{i} gave deterministic primality tests using algebraic tori and elliptic curves.
Here we present a general framework that encompasses and extends the results of \cite{assw,denommesavin08,gross04,tsumura11,AngelaThesis} to arbitrary elliptic curves with CM; this framework is developed in \S \ref{framework} below.

In \cite{assw} we used an elliptic curve with CM by $\Q(\sqrt{-7})$ to obtain an $\tilde{O}(n^2)$ deterministic primality-proving algorithm for a sequence of numbers to which classical methods do not apply.
With this algorithm we were able to prove the primality of a number that, at the time it was found, was the largest proven prime $N$ for which no significant partial factorization of $N-1$ or $N+1$ is known (see also \cite{AngelaThesis}).
Here we use an elliptic curve with CM by $\Q(\sqrt{-15})$, which has class number 2, to obtain an $\tilde{O}(n^2)$ deterministic primality proving algorithm for a new sequence of integers to which classical methods do not apply; this is presented in \S\ref{example-15}.
For elliptic curves with CM by fields of class numbers one and two, these two examples are essentially the only such results possible using our techniques, as explained in \S\ref{splitsect}.

With this new algorithm, we are able to prove the primality of a 1,392,250-bit integer, a number with some 419,110 decimal digits.
As of this writing, this is the largest prime whose primality cannot feasibly be proved using classical methods.

The organization of this paper as follows.
We begin in \S\ref{lemmasect} by  proving a number of lemmas that are used in later sections.
In \S\ref{framework} we state and prove our main theoretical results, and give general algorithms.
In \S\ref{splitsect} we explain why our algorithms in \cite{assw} and
\S\ref{example-15} below are essentially the only ``interesting'' cases that use
elliptic curves with CM by fields of class numbers one or two.
In \S \ref{example-2} we briefly discuss a primality test that uses 
an elliptic curve with CM by $\Q(\sqrt{-2})$; while the sequence of numbers to which it applies can also be tested using classical methods, it represents a case not previously addressed with elliptic curves.
Finally, in \S \ref{example-15} we present our new primality test using an elliptic curve with CM by $\Q(\sqrt{-15})$ and present computational results.

\smallskip 

\noindent{\bf{Acknowledgments:}} 
We thank Daniel J.~Bernstein, Robert Denomme, Fran\c{c}ois Morain, 
Carl Pomerance, and Karl Rubin for helpful conversations,
and the organizers of ECC 2010, the First Abel Conference, and 
the AWM Anniversary Conference 
where useful discussions took place.
We thank the referee for helpful comments.

\section{Lemmas}
\label{lemmasect}

\begin{definition}
Suppose  
$E$ is an elliptic curve over a number field $M$ 
and $J$ is an ideal of $\O_M$ that is prime to $\disc(E)$. 
We say that $P\in E(M)$ is 
{\bf strongly nonzero} mod $J$ if 
one can express $P = (x:y:z) \in E(\O_M)$ in such a way that the ideal
generated by $z$ and $J$ is $\O_M$ (i.e.,  $(z)$ and $J$ are
relatively prime).
We say that $P$ is {\bf nonzero} mod $J$, and write $P \not\equiv O_E$ mod~$J$, if 
one can express $P = (x:y:z) \in E(\O_M)$ in such a way that 
$z\not\in J$; otherwise we say $P$ is zero mod $J$ and write $P \equiv O_E$ mod~$J$.
\end{definition}
\smallskip

\begin{remark}~\\\vspace{-12pt}
\begin{enumerate}
\item
If $P \equiv O_E$ mod~$J$, then $P \equiv O_E$ mod~$\lambda$ for every prime
ideal $\lambda$ that divides $J$ in $\O_M$.
\item 
The point $P$ is strongly nonzero mod $J$ if and only if 
$P \not\equiv O_E$  mod $\lambda$ for every prime ideal
$\lambda$ that divides $J$ in $\O_M$.
\item In particular,
if $J$ is prime, then $P$ is strongly nonzero mod $J$ if and only if $P \not\equiv O_E$ mod~$J$.
\end{enumerate}
\end{remark}

The next result is a direct generalization of Theorem 2 of
\cite{LenstraICM}.

\begin{theorem}
\label{Lenstra}
Suppose $E$ is an elliptic curve over a number field $M$ with CM by the 
ring of integers in an imaginary quadratic field $K$. 
Suppose ${J}$ is an ideal of $\O_M$ relatively prime to $\text{disc(E)}$.
Suppose $Q \in E(M)$, $\Lambda \in \O_K$,
$\Lambda Q \equiv O_E$ mod ${J}$, and for every prime ideal $\lambda$ of $\OK$
that divides $\Lambda$ there exists a point 
in $\frac{(\Lambda)}{\lambda} Q$ that is strongly nonzero mod ${J}$.
Let $r = \norm_{M/\Q}({J})$, and suppose $\norm_{K/\Q}(\Lambda) > \left(r^{1/4} + 1\right)^2$.  
Then the ideal ${J}$ is prime.
\end{theorem}

\begin{proof}
If $J$ is not prime then there is a prime ideal $\q$ of $\O_M$ that divides $J$ 
and satisfies $q:=\norm_{M/\Q}(\q) \le \sqrt{r}$. 
The annihilator of $Q$ mod $\q$ is $(\Lambda)$, so 
the $\OK$-module generated by $Q$ mod $\q$ is isomorphic to $\OK/(\Lambda)$.
By the Hasse bound,
$$(\sqrt{q} + 1)^2 \ge |E(\O_M/\q) | \geq \norm_{K /\Q}(\Lambda)  
> \left(r^{1/4} + 1\right)^2 \ge (\sqrt{q} + 1)^2.$$
This contradiction implies that $J$ is prime.
\end{proof}

The next result will be used in the proof of Theorem \ref{mainthm2}.

\begin{lemma}
\label{Mxmodule}
Suppose $K$ is an imaginary quadratic field, and $\a$ is an ideal of $\O_K$
such that $\a$ is not divisible in $\O_K$ by any rational prime that
splits in $K$. Define $L\in\Z^+$ by $L\Z=\a\cap\Z$.
Suppose $\m$ is a finite $\O_K$-module, and $x$ is an element of $\m$ of order $L$.
Then
$$
\#(\O_K\cdot x) \ge {\frac{N_{K/\Q}(\a)}{N_{K/\Q}(\prod_\lambda \lambda)}},
$$
where $\lambda$ runs over the prime ideals of $\O_K$ that divide
$\a$ and are ramified in $K/\Q$.
\end{lemma}

\begin{proof}
Let $\O=\O_K$.
First suppose $\a = \lambda^n$ with $\lambda$ a prime ideal of $\O$.
Let $p$ be the rational prime below $\lambda$.

If $p=\lambda\bar{\lambda}$ splits in $K$, then $L=p^n$,
and the cyclic $\O$-module $\O\cdot x$ is isomorphic to
$\Z/p^a\Z \times \Z/p^b\Z$, with $\max\{a,b\}=n$.
Then
$$
\#(\O\cdot x) =p^{a+b} \ge p^n = N_{K/\Q}(\a)
$$
as desired.

If $p=\lambda$ is inert, then $L=p^n$
and $\O\cdot x \simeq \O/\lambda^n \simeq \O/(p^n)$, so
$$
\#(\O\cdot x) =\#(\O/(p^{n})) = p^{2n} = N_{K/\Q}(\a).
$$

If $p=\lambda^2$ ramifies, then $\a\cap\Z=\lambda^n\cap\Z
=p^{\lceil n/2 \rceil}\Z$ so
$L=p^{\lceil n/2 \rceil}$.
The cyclic $\O$-module $\O\cdot x$ is isomorphic to
$\O/\lambda^a$, where $1$ has order $L=p^{a/2}$ if $a$ is even
and order $L=p^{(a+1)/2}$ if $a$ is odd.
Thus $\#(\O\cdot x) =\#(\O/\lambda^a) =p^{a}$ is $L^2$ if $a$ is even
and $L^2/p$ otherwise. 
Thus
$$
\#(\O\cdot x) \ge L^2/p \ge p^{n}/p = N_{K/\Q}(\a)/N_{K/\Q}(\lambda)
$$
as desired.

Let us now consider an ideal $\a$ satisfying the hypothesis of the lemma.
Let $\b$ be the annihilator of $x$ in $\O$.
We have $\O\cdot x \simeq \O/\b \simeq \oplus_{\lambda}\O/\lambda^{n_\lambda}$,
with a finite direct sum with $\lambda$'s distinct prime ideals of~$\O$.
Define $L_p\in\Z^+$ by
$$
L_p\Z = \bigl(\prod_{\lambda\mid p} \lambda^{n_\lambda}\bigr)\cap\Z
 = (\cap_{\lambda\mid p} \lambda^{n_\lambda})\cap\Z.
 $$
Then $L_p$ is a power of $p$, so the $L_p$'s are relatively prime.

We now show that $L = \prod_p L_p$.
Since $L$ is the order of $x$ and $\b$ is the annihilator of $x$, we have $$
L\Z = \b\cap\Z
= (\bigcap_p\prod_{\lambda\mid p} \lambda^{n_\lambda})\cap\Z
= \bigcap_p(\prod_{\lambda\mid p} \lambda^{n_\lambda}\cap\Z)
= \bigcap_p L_p\Z
= (\prod_{p} L_p)\Z.
$$
Write $\a=\prod_\lambda \lambda^{m_\lambda}$, let
$\a_p=\prod_{\lambda\mid p} \lambda^{m_\lambda}$, and define
$M_p\in\Z^+$ by $M_p\Z = \a_p\cap\Z$.
We claim that $L_p=M_p$, as follows.
Since $L\Z = \a\cap\Z$, the proof above, with $\b$ replaced by $\a$
and $n_\lambda$ replaced by $m_\lambda$, shows that $L=\prod_p M_p$.
Thus, both $M_p$ and $L_p$ are the $p$-part of $L$.

Now $\O/\b \simeq \bigoplus_{p}(\O/\prod_{\lambda\mid p}\lambda^{n_\lambda})$.
Since $\a$ is not divisible in $\O$ by any rational prime that
splits in $K$, we have
$\a_p=\prod_{\lambda\mid p} \lambda^{m_\lambda}
=\lambda^{m_\lambda}$, for some $\lambda\mid p$, so $\a_p$ is a prime power.
Thus (applying the prime power case to the element $1\in
\O/\prod_{\lambda\mid p}\lambda^{n_\lambda}$ which has order $L_p$),
$$
\#(\O\cdot x) = \#(\O/\b) =
\prod_{p}\#(\O/\prod_{\lambda\mid p}\lambda^{n_\lambda})
\ge 
\prod_{p}{\frac{N_{K/\Q}(\a_p)}{N_{K/\Q}(\prod_{\lambda\mid \a_p} \lambda)}}
={\frac{N_{K/\Q}(\a)}{N_{K/\Q}(\prod_{\lambda\mid \a} \lambda)}},
$$
where in the latter two terms $\lambda$ runs over prime ideals
that are ramified in $K/\Q$.
\end{proof}

The following lemma is completely elementary, and will be used
in the proofs of Theorems \ref{mainthm1} and \ref{mainthm2}.

\begin{lemma}
\label{assm5}
If $C, \f\in\R$, $C\ge 1$, and $\f > 16C^2$, 
then $C(\f^{1/4} + 1)^2 < (\sqrt{\f} - 1)^2.$
\end{lemma}

\begin{proof}
Since $16<\f$ we have $16\f < \f^2$, so $2\f^{1/4} < \sqrt{\f}$.
Since $\f > 16C^2$, we have $\sqrt{\f} > 4C$, so 
$2\sqrt{\f}  < \frac{\f}{C} - 2\sqrt{\f}$. Thus,
\begin{align*}
(\f^{1/4} + 1)^2 &=\sqrt{\f} + 2\f^{1/4} + 1 < 2\sqrt{\f} + 1\\
&<\frac{\f}{C} - 2\sqrt{\f} + 1 \le  \frac{\f}{C} - \frac{2\sqrt{\f} - 1}{C} = 
\frac{(\sqrt{\f} - 1)^2}{C}.\qedhere
\end{align*}
\end{proof}

\section{Main Theoretical Results and Proofs}
\label{framework}

\subsection{Notation and assumptions}
\label{assmptnsect}
Let $K$ be an imaginary quadratic field, let
$E$ be an elliptic curve over a number field $\M \supseteq K$ with CM by $\OK$, 
and let $P$ be a point in $E(M)$.
Let $\gamma, \alpha_1 , \ldots \alpha_s$ be nonzero elements of $\OK$, and for $k = (k_1,\ldots, k_s)\in\N^s$ define
$$
\Lambda_k = \gamma\alpha_1^{k_1}\cdots\alpha_s^{k_s}, \qquad \pi_k = 1 + \Lambda_k, \qquad F_k = \norm_{K / \Q}(\pi_k).
$$
We shall restrict our attention to $F_k > 16\norm_{K/\Q}(\gamma^2)$ such that
$\text{disc}(E)$ is relatively prime to $F_k$, which we assume henceforth (this is simply a restriction on $k$).
Let $\frakp_k$ denote an ideal of $\O_{\M}$ such that 
$N_{\M/K}(\frakp_k) = (\pi_k)$, so $F_k=N_{\M /\Q}(\frakp_k)$.
Further, let us assume that $k$ is chosen so that
whenever $\frakp_k$ is prime the following hold:
\begin{enumerate}
\item
the Frobenius endomorphism of $E$ over 
the field $\O_\M/\frakp_k$ is $\pi_k$;  
\item we have $P\bmod {\frakp_k} \not\in \lambda E(\O_{\M}/\frakp_k)$
for every prime ideal $\lambda$ of $\OK$ that divides $(\alpha_1 \cdots \alpha_s)$.

\end{enumerate}
In applications, we will make explicit choices for $K$, $E/M$, and~$P$, and then determine explicit arithmetic conditions on $k$ that ensure that the above assumptions hold.
Here we state our assumptions generically for the purpose of proving our main theoretical results, Theorems~\ref{mainthm1} and~\ref{mainthm2} below.

\begin{definition}
With notation as above, 
define $L_k\in\Z^+$ by $L_k\Z = {\frac{\Lambda_k}{\gamma}}\OK\cap\Z$.
\end{definition}

\begin{remark}
If $p$ is a prime, then $p\mid\norm_{K / \Q}(\alpha_1 \cdots \alpha_s)$ 
if and only if $p\mid L_k$, since 
$\frac{\Lambda_k}{\gamma} = \alpha_1^{k_1}\cdots\alpha_s^{k_s}$
and $L_k\Z = \left(\frac{\Lambda_k}{\gamma}\right)\cap\Z$.
\end{remark}

\begin{remark}
Whenever $p$ is a rational prime,
let $n_p := \ord_p(\prod_{i=1}^s N_{K/\Q}(\alpha_i)^{k_i})$ and let
$m_p := n_p$ if $p$ splits in $K$ and 
$m_p := \lceil n_p/2 \rceil$ otherwise.
If $(\alpha_1\cdots\alpha_s)$ is not divisible by any
rational prime that splits in $K$ then $L_k = \prod_p p^{m_p}$.
\end{remark}

\begin{lemma}
\label{NormIneq}
With notation as above we have
$$
N_{K/\Q}(\Lambda_k) \ge (\sqrt{F_k}-1)^2.
$$
\end{lemma}

\begin{proof}
We have
\begin{align*}
N_{K/\Q}(\Lambda_k) &= N_{K /\Q}(\pi_k - 1) = (\pi_k - 1)(\overline{\pi_k} - 1) = \pi_k\overline{\pi_k} - 2\text{Re}(\pi_k) + 1\\
&\ge N_{K/\Q}(\pi_k) - 2\sqrt{N_{K/\Q}(\pi_k)} + 1 = (\sqrt{F_k}-1)^2.\qedhere
\end{align*}
\end{proof}

Whenever the $\O_\M$-ideal $\frakp_k$ is a prime ideal and
the Frobenius endomorphism of~$E$ over $\O_\M/\frakp_k$ is $\pi_k$,
then $\Lambda_k$ is separable as an endomorphism of $E$ mod $\frakp_k$;
this follows from Corollary 5.5 in Chapter III of \cite{silverman09}.

\subsection{Main Theoretical Results}\label{mainresults}

\begin{theorem}
\label{mainthm1}
With notation and assumptions as above, the following are equivalent:
\begin{enumerate}
\item[\rm (a)] the ideal $\frakp_k$ is prime;
\item[\rm (b)] $\Lambda_k P \equiv O_E$ mod $\frakp_k$, and for every prime ideal $\lambda \mid (\alpha_1\cdots\alpha_s)$ of $\O_K$ there is a point in $\frac{(\Lambda_k)}{\lambda}P$ that is  strongly nonzero mod $\frakp_k$.
\end{enumerate}
\end{theorem}

\begin{proof}
Assume $\frakp_k$ is a prime ideal. 
Since we assumed that the Frobenius endomorphism of $E$ over 
$\O_\M/\frakp_k$ is $\pi_k$, as $\OK$-modules we have 
$$
E(\O_{\M}/\frakp_k) \simeq  \ker(\pi_k - 1) \simeq  \OK/(\pi_k - 1) = \OK/(\Lambda_k).                  
$$
Thus, $\Lambda_k P \equiv O_E$ mod $\frakp_k$.
Suppose $\lambda$ is a prime ideal of $\OK$ that divides $(\alpha_1 \cdots \alpha_s)$.
Since $E(\O_{\M}/\frakp_k)\simeq  \OK/(\Lambda_k)$,
the points in $\lambda E(\O_{\M}/\frakp_k)$ are exactly the points
in $E(\O_{\M}/\frakp_k)$ that are killed by $\frac{(\Lambda_k)}{\lambda}$.
Assumption (ii) now gives (b).

Conversely, suppose that
$\Lambda_k P \equiv O_E$ mod $\frakp_k$, and that for every prime ideal $\lambda \mid \left(\alpha_1\cdots\alpha_s\right)$ there is a point in $\frac{(\Lambda_k)}{\lambda}P$ that is  strongly nonzero mod $\frakp_k$.
Since $F_k > 16\norm_{K/\Q}(\gamma^2)$, by Lemma \ref{assm5} with
$F = F_k$ and $C = \norm_{K/\Q}(\gamma)$ and 
Lemma \ref{NormIneq}
we have 
$N_{K/\Q}(\Lambda_k/\gamma) > (F_k^{1/4}+1)^2$.
Apply Theorem \ref{Lenstra} with 
$J= \frakp_k$, $\Lambda = \Lambda_k/\gamma$, and $Q=\gamma P$
to conclude that $\frakp_k$ is prime.
\end{proof}

\begin{theorem}
\label{mainthm2}
With notation and assumptions as above, suppose further that  
$(\alpha_1\cdots\alpha_s)$ is not divisible in $\OK$ by any
rational prime that splits in $K$,
and that
$F_k > 16\norm_{K/\Q}(\gamma\prod_\lambda \lambda)^2$
where $\lambda$ runs over the prime ideals of $\O_K$ that divide
$(\alpha_1\cdots\alpha_s)$ and are ramified in $K/\Q$.
Then the following are equivalent:
\begin{enumerate}
\item[\rm (a)] the ideal $\frakp_k$ is prime;
\item[\rm (b)] $L_k\gamma P \equiv O_E$ mod $\frakp_k$, and 
$\frac{L_k}{p}\gamma P$ is strongly nonzero mod $\frakp_k$
 for every prime divisor $p$ of
$\norm_{K / \Q}(\alpha_1 \cdots \alpha_s)$.
\end{enumerate}
\end{theorem}

\begin{proof}
Suppose $\frakp_k$ is prime. By Theorem \ref{mainthm1}, we have 
$\Lambda_k P \equiv O_E$ mod $\frakp_k$, and 
$\frac{(\Lambda_k)}{\lambda}P \not\equiv O_E$  mod $\frakp_k$
for every prime ideal $\lambda \mid (\alpha_1\cdots\alpha_s)$.
Let
$$
\a := \{ \beta\in\OK : \beta(\gamma P)\equiv O_E \bmod{\frakp_k}\} \supseteq \left(\frac{\Lambda_k}{\gamma}\right).
$$
If $\a \neq \left(\frac{\Lambda_k}{\gamma}\right)$ then 
$\a\lambda \supseteq \left(\frac{\Lambda_k}{\gamma}\right)$ for some prime ideal
$\lambda$ of $\O_K$. 
Thus,  
$$
\lambda^{-1}\left(\frac{\Lambda_k}{\gamma}\right)(\gamma P) \equiv O_E\bmod \frakp_k,
$$
that is, 
$\frac{(\Lambda_k)}{\lambda}P \equiv O_E$ mod $\frakp_k$.
Since the prime ideal $\lambda$ divides $\bigl(\frac{\Lambda_k}{\gamma}\bigr)$
it divides $(\prod \alpha_i)$, which is a contradiction. 
Thus, $\a = \bigl(\frac{\Lambda_k}{\gamma}\bigr).$

Since $L_k\in\Z^+$ and $L_k\Z = \left(\frac{\Lambda_k}{\gamma}\right)\cap\Z$
we have: 
\begin{enumerate}[(i)]
\item
$L_k\gamma\in (\Lambda_k)$ so $L_k\gamma P \equiv O_E$ mod $\frakp_k$;
\item for all primes $p\mid L_k$ we have 
$\frac{L_k}{p} \not\in \left(\frac{\Lambda_k}{\gamma}\right)$.
\end{enumerate}
Thus $\frac{L_k}{p}(\gamma P) \not\equiv O_E$ mod $\frakp_k$, and we have shown that (a) implies (b).

We now assume (b) and suppose, for the sake of obtaining a contradiction, that $\frakp_k$ is not prime.
Then there is a prime ideal $\q$ of $\O_{\M}$ dividing $\frakp_k$ such that 
$q := N_{\M/\Q}(\q) \le N_{\M/\Q}(\frakp_k)^{1/2} = \sqrt{F_k}$.
Next we will apply Lemma~\ref{Mxmodule} with $\m=E(\O_{\M}/\q)$,
$x= \gamma P$ mod $\q \in \m$, and 
$\a = \left(\frac{\Lambda_k}{\gamma}\right) = (\alpha_1^{k_1}\cdots\alpha_s^{k_s})$
(using our assumption that
$(\alpha_1\cdots\alpha_s)$ is not divisible in $\OK$ by any rational prime that splits in $K$).
Now (b) implies that $L_k$ is the order of $x$ in $\m$.
By Lemmas \ref{Mxmodule} and \ref{NormIneq}, and the Hasse bound, we have
\begin{align*}
(F_k^{1/4}+1)^2 &\ge (\sqrt{q}+1)^2 \ge |E(\O_{\M}/\q)|\\
&\ge \#\OK\cdot x \ge 
{\frac{N_{K/\Q}(\frac{\Lambda_k}{\gamma})}{N_{K/\Q}(\prod_\lambda \lambda)}}
\ge \frac{(\sqrt{F_k}-1)^2}{N_{K/\Q}(\gamma\prod_\lambda \lambda)},
\end{align*}
where $\lambda$ runs over the prime ideals of $\O_K$ that divide
$(\alpha_1 \cdots \alpha_s)$ and are ramified in $K/\Q$.
Since $F_k > 16\norm_{K/\Q}(\gamma\prod_\lambda \lambda)^2$,
we can apply Lemma \ref{assm5} with $C=\norm_{K/\Q}(\gamma\prod_\lambda \lambda)$
to obtain the desired contradiction.
\end{proof}

\begin{remark}
If $F_k$ is prime, then $\frakp_k$ is a prime ideal.
If $\frakp_k$ is a prime ideal, then $F_k$ is a prime power $p^r$
where $r\mid 2h$ and $h$ is the class number of $K$.
Thus, Theorems \ref{mainthm1} and \ref{mainthm2} are
essentially primality tests for $F_k$.
\end{remark}

\subsection{Algorithms}

The following two algorithms are deterministic primality tests for the sequence $F_k$.
That they produce correct outputs follows from Theorems
\ref{mainthm1} and \ref{mainthm2}.
In the algorithms, points on $E$ and $E$ mod $\frakp_k$ are expressed 
in primitive projective coordinates, as in \S 3 of \cite{LenstraICM}.
Thus $P = [x:y:z]\in E(\M)$ implies that $x,y,z\in \O_\M$ and that the ideal $(x,y,z)$ is equal to $\O_\M.$  

Algorithm \ref{primetest2} takes as input the data in the hypotheses of Theorem \ref{mainthm1}.  

\begin{algorithm}
\label{primetest2}~\\\vspace{-12pt}
\begin{enumerate}[\rm 1.] 
\item Compute $\tilde{P} := P$ mod ${\frakp_k}$ and 
$\Lambda_k\tilde{P} = [x' : y' : z'] \in E(\O_{\M} / \frakp_k)$.  
\item If $z'  \neq 0$ in $\O_{\M} / \frakp_k$, output ``$\frakp_k$ is not a prime ideal"  and terminate.
\item 
For every prime ideal $\lambda | \left( \prod \alpha_i\right)$: 
\begin{enumerate}[\rm a.]
\item
Choose $\beta_{\lambda} \in \frac{(\Lambda_k)}{\lambda} - (\Lambda_k)$ and 
compute $\beta_{\lambda}\tilde{P} = [x_{\lambda} : y_{\lambda} : z_{\lambda}] \in E(\O_{\M} / \frakp_k)$.
\item  
Choose  $Z_{\lambda} \in \O_{\M}$ with $Z_{\lambda} \equiv z_{\lambda} \bmod{\frakp_k}$
and compute the ideal\\
$\Delta_{\lambda} := (Z_{\lambda}, \frakp_k)$ of $\O_{\M}$.
\end{enumerate}
\item If there is a $\lambda$ with $\lambda | \left( \prod \alpha_i\right)$ and $\Delta_{\lambda} \neq \O_\M$, output ``$\frakp_k$ is not a prime ideal,"
otherwise, output ``$\frakp_k$ is a prime ideal."
\end{enumerate}
\end{algorithm}

Algorithm \ref{primetestmod} takes as input the data in the hypotheses of Theorem \ref{mainthm2}.

\begin{algorithm}
\label{primetestmod}~\\\vspace{-12pt}
\begin{enumerate}[\rm 1.] 
\item Compute $Q := \gamma P$ mod ${\frakp_k}$. 
\item
For every $p | \norm_{K / \Q}( \prod \alpha_i)$:
\begin{enumerate}[\rm a.] 
\item
Compute $P_{p} := \frac{L_k}{p}Q = [x_p : y_p : z_p] \in E(\O_{\M} / \frakp_k)$.  
\item 
Choose $Z_{p} \in \O_{\M}$ with $Z_p \equiv z_p \bmod{\frakp_k}$
and compute the ideal\\
$\Delta_{p} := (Z_{p} , \frakp_k)$ of $\O_{\M}$.
\end{enumerate}
\item If there is a prime $p | \norm_{K / \Q}( \prod \alpha_i)$ such that 
$\Delta_{p} \neq \O_\M$,\\
then output ``$\frakp_k$ is not a prime ideal'' and terminate.
\item Compute $L_k Q = [x':y':z'] \in E(\O_{\M} / \frakp_k)$.
\item If $z' \neq 0$ in $\O_{\M} / \frakp_k$, output 
``$\frakp_k$ is not a prime ideal,''\\
otherwise, output ``$\frakp_k$ is a prime ideal."
\end{enumerate}
\end{algorithm}

\section{Constraints on $K$}
\label{splitsect}

In order to obtain efficient deterministic primality tests
we need to determine in advance the ``good" $k$, that is, the values of
$k$ that satisfy assumptions (i) and (ii) of \S\ref{assmptnsect} whenever the
ideal $\frakp_k$ is prime.
Satisfying assumption (i) is not a problem.
Finding a nice set of $k$ that satisfy assumption (ii) is more problematic.

In Theorem \ref{splittingiff} below we show that assumption (ii), which states that
$P\bmod \mathfrak{p}_k \not\in\lambda E(\O_\M/\mathfrak{p}_k)$ for every prime $\lambda$ dividing $(\alpha_1\cdots \alpha_s)$,
is equivalent to the prime ideal $\frakp_k$
splitting completely in $F$ but not in~$L$, for certain extension fields
$F$ and $L$ with 
$\M \subseteq F \subseteq L$.
In particular, it is necessary to have $F \neq L$. 
When the extension $L/\M$ is not abelian, 
we do not know a good way to characterize the prime ideals of $\O_\M$ that split completely in $F$ but not in~$L$. 
However, if $L/\M$ is abelian, then 
class field theory implies that the splitting behavior 
in $L$ and $F$ of a prime of $\O_\M$ is determined by congruence conditions.
So in order to obtain useful algorithms
we insist that $L/\M$ be abelian and that $F \neq L$.
We show in Lemma \ref{abelianeither1} below that this forces $F=\M$, 
and this in turn is equivalent to 
$E[\lambda] \subseteq E(\M)$.
This severely restricts the possibilities for $K$ and $\lambda$
(see Theorem \ref{abeliansplit2} below).
In Theorems \ref{fewexamples} and \ref{classno2examples} 
we determine the exact possibilities
for $K$ and $\lambda$ when $K$ has class number one or two and $E$
is defined over $\Q(j(E))$.

We begin with some preliminaries. 
Suppose $K$ is an imaginary quadratic field, 
$E$ is an elliptic curve over a field $M$
(not necessarily a number field) with CM by $\O_K$, 
the endomorphisms of $E$ are all defined over $M$, 
$\lambda$ is a prime ideal of $\O_K$ such that $\ch(M)\nmid N_{K/\Q}(\lambda)$,
and $p$ is the rational prime below $\lambda$.
Let  $E'=E/E[\bar{\lambda}]$, let 
$\hat{\varphi}: E \to E'$ denote the natural isogeny, and 
let $\varphi: E' \to E$ denote the dual isogeny. 

\begin{lemma}
\label{EEprimeLem}
With notation as above, $M(E[{\lambda}]) = M(E'[{\lambda}])$.
\end{lemma}

\begin{proof}
There are an $\OK$-ideal $\a$ and an element $\beta\in\OK$
such that $(\beta)=\a \overline{\lambda}$ and
$N_{K/\Q}(\a)$ is not divisible by $\ch(M)$ or by $N_{K/\Q}(\lambda)$.
Since $\ch(M)$ does not divide $N_{K/\Q}(\lambda)N_{K/\Q}(\beta)$, the isogenies $\hat{\varphi}$ 
and $\beta$ are separable.
Since
$$\ker(\hat{\varphi}) = \ker(\overline{\lambda}) \subseteq \ker(\beta),$$
there is a unique isogeny $\psi : E'\to E$ such that
$\beta = \psi\circ\hat{\varphi}$ (see Corollary~4.11 in \cite[\S III.4]{silverman09}).
Then $\psi$ is defined over $M$.
Since $\ker(\hat{\varphi})=E[\overline{\lambda}]$, 
$\hat{\varphi}$ is injective on $E[\a]$, we have
$E'[\a] = \hat{\varphi}(E[\a]) \subseteq \ker(\psi)$.
Both $E'[\a]$ and $\ker(\psi)$ have order $N(\a)$, so $\ker(\psi)=E'[\a]$.
Thus $\ker(\psi)\cap E'[\lambda] =0$,
hence $\psi$ induces an $M$-isomorphism 
from $E'[{\lambda}]$ onto $E[{\lambda}]$ as desired.
\end{proof}
\smallskip

\begin{remark}~\\\vspace{-12pt}
\label{PQwelldefrmk}
\begin{enumerate}
\item
If $\lambda$ is a principal ideal $(\alpha)$ of $\O_K \subseteq\End(E)$
(e.g., if $p$ is inert in $K/\Q$)
then $E$ and $E'$ are isomorphic over $M$,
and for all
$P\in E(M)$ we have 
$M(\varphi^{-1}(P)) = M(\alpha^{-1}(P))$.
\item 
If $P\in E(\overline{M})$ and 
$Q_1, Q_2\in E'(\overline{M})$ satisfy $\varphi(Q_i)=P$,
then $Q_1 - Q_2 \in \ker(\varphi) =   
E'[\lambda]$. Thus 
$M(E'[\lambda],Q_1)=M(E'[\lambda],Q_2)$, 
so 
$M(E'[\lambda],\varphi^{-1}(P))$ is well-defined.
\end{enumerate}
\end{remark}

\begin{lemma}
\label{fandlambda}
With notation as above, $\varphi(E'(M)) = \lambda E(M)$.
\end{lemma}

\begin{proof}
This is clear if $p$ is inert in $K$, so assume
$p$ is split or ramified.
Define $G_M:= \Gal(\bar{M}/M)$.
By the definition of $E'$, we have
$E'(\bar{M})= E(\bar{M})/E[\bar{\lambda}]$. 
It follows that
$$
E'(M) = E'(\bar{M})^{G_M} = 
\{ \hat{\varphi}(R) : R \in E(\bar{M}) \text{ and $\sigma(R)-R\in E[\bar{\lambda}], \medspace \forall \sigma\in G_M \}$}.
$$

Suppose $\hat{R}\in E'(M)$, that is, 
$\hat{R}=\hat{\varphi}(R) = R + E[\bar{\lambda}]$, with $R \in E(\bar{M})$ and
$\sigma(R)-R\in E[\bar{\lambda}]$ for all $\sigma\in G_M$.
Then $\bar{\lambda}R \subset E(M)$ (since if $\beta\in\bar{\lambda}$ then we have
$\sigma(\beta R)-\beta R = \beta(\sigma R-R)=O_{E'}$ for all $\sigma\in G_M$).
Since $\varphi \circ \hat{\varphi} = p$, we have 
$\varphi(\hat{R}) = pR\in p\O_K R = \lambda\bar{\lambda}R \subset \lambda E(M)$.

Conversely, suppose $P\in\lambda E(M)$, i.e.,
$P = \sum \alpha_i Q_i$ with $\alpha_i\in \lambda$ and $Q_i\in E(M)$.
Take $S_i\in E(\bar{M})$ so that $pS_i=Q_i$, and let
$R = \sum \alpha_i S_i$.
If $\beta\in\bar{\lambda}$ and $\sigma\in G_M$, then
$\beta(\sigma R-R) =  \sum \beta\alpha_i(\sigma S_i - S_i) = O_{E'}$
(since $\beta\alpha_i\in \bar{\lambda}{\lambda}=p\O_K$, and 
all the endomorphisms of $E$ are defined over $M$ so 
$\sigma S_i - S_i \in E[p]$).
Thus, $\sigma R-R \in E[\bar{\lambda}]$, so
$\hat{\varphi}(R) \in E'(M)$.
Since $\varphi(\hat{\varphi}(R)) = pR = \sum \alpha_i Q_i = P$, 
we have $P\in \varphi(E'(M))$.
\end{proof}

\begin{theorem}
\label{splittingiff}
Suppose $K$ is an imaginary quadratic field,
$\lambda$ is a prime ideal of $\O_K$,
$E$ is an elliptic curve over a number field $\M$ with CM by $\O_K$, 
$\p$ is a prime ideal of~$\O_\M$, 
and
$P \in E(\M)$.
Suppose   
$\p \nmid N_{K/\Q}(\lambda)\Delta(E)$, where $\Delta(E)$ is the 
discriminant of $E$.
Let
$E'=E/E[\bar{\lambda}]$,  let
$\hat{\varphi}: E \to E'$ denote the natural isogeny, and
let $\varphi: E' \to E$ denote the dual isogeny.
Let $\tilde{E}$ and $\tilde{P}$ denote the reductions modulo $\p$ of
$E$ and $P$, respectively.
Let $F:= \M(E'[\lambda])$ and $L:= F(\varphi^{-1}(P))$.
Then the following are equivalent:
\begin{enumerate}
\item
$\tilde{P} \not\in \lambda \tilde{E}(\O_\M/\p)$;
\item 
$\p$ splits completely in $F$ and 
$\p$ does not split completely in $L$.
\end{enumerate}
\end{theorem}

\begin{proof}
By Lemma \ref{EEprimeLem} we have $F=\M(E'[\lambda])=\M(E[\lambda])$.
Since $\p \nmid N_{K/\Q}(\lambda)\Delta(E)$, $\p$ does not ramify in $F$.
Let $\p'$, respectively $\p''$, denote a prime of $F$, respectively $L$,
above $\p$. Let $k= \O_\M/\p$,
$k'=\O_F/\p'$, and $k''=\O_L/\p''$.
Let $\tilde{E'}$ denote the reduction mod $\p$ of
$E'$.
Since $\p \nmid N_{K/\Q}(\lambda)\Delta(E)$, we have
$k'=k(\tilde{E'}[\lambda])=k(\tilde{E}[\lambda])$, by Lemma \ref{EEprimeLem}, 
$k''=k'(\varphi^{-1}(\tilde{P}))$,
and $\p''/\p$ is unramified.

Let us first suppose that $\p$ splits completely in $F$, meaning that $k'= k$.
Then $\p$ splits completely in $L$ if and only if
$k''=k$, equivalently, if and only if we have $\varphi^{-1}(\tilde{P})\in \tilde{E'}(k)$.
By Lemma \ref{fandlambda},
this holds 
if and only if $\tilde{P}\in \lambda \tilde{E}(k)$.

By the theory of complex multiplication, 
there is an isomorphism of $\O_K$-modules
$\tilde{E}(k) \simeq \O_K/\mathfrak{a}$, for some ideal $\mathfrak{a}$ of $\O_K$.
Now suppose $\p$ does not split completely in $F$, so
$k'\neq k$ and $\tilde{E}[\lambda] \not\subset \tilde{E}(k)$.
It then follows that $\lambda\nmid\mathfrak{a}$, so
$\lambda(\O_K/\mathfrak{a})=\O_K/\mathfrak{a}$.
Thus $\lambda \tilde{E}(k)=\tilde{E}(k)$, and therefore $\tilde{P}\in \lambda \tilde{E}(k)$.
\end{proof}

\begin{lemma}
\label{abelianeither1}
With notation as in Theorem \ref{splittingiff},
$L/\M$ is abelian if and only if either $L=F$ or $F=\M$.
\end{lemma}

\begin{proof}
Let $\Gamma = \Gal(L/F)$ and $G = \Gal(F/\M).$
If we choose $Q \in E'(L)$ so that $\varphi(Q)=P$, then $L=F(Q)$, by Remark \ref{PQwelldefrmk}.
The map that sends each $\tau\in \Gamma$ to $\tau(Q)-Q\in\ker(\varphi)=E'[\lambda] \simeq \O_K/\lambda$ 
is an injective group homomorphism 
$$
\kappa : \Gamma \hookrightarrow E'[\lambda] \simeq \O_K/\lambda.
$$ 
Thus, $\Gamma$ is abelian.
The definition of $G$ gives an injective group homomorphism
$$
\omega : G \hookrightarrow \Aut(E[\lambda]) \simeq (\O_K/\lambda)^\times
$$
with the property that $s(R) = \omega(s)R$ for all
$R\in E'[\lambda]$ and $s\in G$.
Thus, $G$ is abelian.
It follows that if either $\Gamma=1$ or $G=1$
(i.e., $L=F$ or $F=\M$) 
then  $L/\M$ is abelian.

Suppose $\tau\in\Gamma$ and $s\in G$. Lift $s$ to $\sigma\in\Gal(L/\M)$.
Then $\sigma\tau\sigma^{-1}\in \Gamma$,
and a straightforward calculation shows that 
$\kappa(\sigma\tau\sigma^{-1}) = \omega(s)\cdot\kappa(\tau)$.
If $s\neq 1$ and $\tau\neq 1$, then $\omega(s) \neq 1$ and 
$\kappa(\tau)\neq 0$, so 
$\kappa(\sigma\tau\sigma^{-1}) \neq \kappa(\tau)$,
and thus $\sigma\tau\sigma^{-1} \neq \tau$.
In other words, if $\Gamma \neq 1$ and $G\neq 1$, then $L/\M$ is not abelian.
\end{proof}

\begin{lemma}
\label{GBlem}
Suppose $K$ is an imaginary quadratic field with Hilbert class field~$H$,
$E$ is an elliptic curve over $H$ with CM by $\O_K$,  
and $B$ is an ideal of $\O_K$.
Let $G_B = \Gal(H(E[B])/H)$.
Then the natural map $G_B \hookrightarrow \Aut(E[B]) \simeq (\O_K/B)^\times$
induces a {\bf surjection}
$$
G_B \twoheadrightarrow \Aut(E[B])/i(\O_K^\times) \simeq (\O_K/B)^\times/i(\O_K^\times),
$$
where $i : \O_K^\times \to \Aut(E[B]) \simeq (\O_K/B)^\times$ is the natural map.
\end{lemma}

\begin{proof}
This follows from Theorem 5.4 of \cite{Shimurabook}.
\end{proof}

\begin{theorem}
\label{abeliansplit2}
Suppose $K$ is an imaginary quadratic field with Hilbert class field~$H$,
$E$ is an elliptic curve over $\Q(j(E)) \subset H = K(j(E))$ with CM by $\O_K$,  
$\lambda$ is a prime ideal of $\O_K$,
$E[\lambda] \subseteq E(H)$, and
$p$ is the rational prime below $\lambda$.
Then one of the following holds:
\begin{enumerate}[\rm (a)]
\item
$p=2$, and $2$ splits in $K$;
\item
$p=2$ or $3$, and $p$ ramifies in $K$;
\item
$K=\Q(\sqrt{-3})$ and $\lambda = (2)$.
\end{enumerate}
\end{theorem}

\begin{proof}
Apply Lemma \ref{GBlem} with $B=\lambda$. 
Since $G_\lambda = \Gal(H(E[\lambda])/H)=1$,
this gives 
\begin{equation}
\label{Olambdaeq}
|\O_K^\times| \ge |(\O_K/\lambda)^\times| =
 |\O_K/\lambda| - 1 = N_{K/\Q}(\lambda)-1.
\end{equation}

If $p$ splits in $K$,
then $(p) = \lambda\overline{\lambda}$ and $\lambda \neq \overline{\lambda}$.
Since $E[\lambda] \subseteq E(H)$ and
$E$ is defined over $\Q(j(E))$, it follows that
$E[\overline{\lambda}] \subseteq E(H)$ 
(if $1 \neq \sigma \in \Gal(H/\Q(j(E)))$, then $\sigma(\lambda) = \overline{\lambda}$,
so $\sigma(E[\lambda]) = E[\overline{\lambda}]$).
Thus, $E[p]  \subseteq E(H)$.
Now take $B=(p) = \lambda\overline{\lambda}$ in Lemma \ref{GBlem}.
We have $H(E[p])=H$, so $G_{\lambda\overline{\lambda}}=1$. 
Now by Lemma \ref{GBlem}, if $p$ splits in $K$ then 
\begin{equation}
\label{pspliteq}
|\O_K^\times| \ge |(\O_K/\lambda\overline{\lambda})^\times| =
(N_{K/\Q}(\lambda)-1)^2 = (p-1)^2.
\end{equation}

First suppose $K \neq \Q(i)$ or $\Q(\sqrt{-3})$. 
Then $|\O_K^\times|=2$.
By \eqref{Olambdaeq}, $N_{K/\Q}(\lambda) \le 3$.
Thus, $\lambda$ is a prime ideal above $p=2$ or $3$, and $p$ splits or
ramifies in $K$.
By \eqref{pspliteq}, if $p$ splits then $p \neq 3$. 

Suppose $K = \Q(i)$.
By \eqref{Olambdaeq}, $N_{K/\Q}(\lambda) \le 5$.
If $p=5$, which splits, then \eqref{pspliteq} gives a contradiction.
The remaining case is when $p=2$, which splits.

Suppose $K = \Q(\sqrt{-3})$.
By \eqref{Olambdaeq}, $N_{K/\Q}(\lambda) \le 7$.
If $p=7$, which splits, then \eqref{pspliteq} gives a contradiction.
The remaining cases are when $p=3$ (which ramifies) or $\lambda = (2)$.
\end{proof}

\begin{theorem}
\label{fewexamples}
Suppose $K$ is an imaginary quadratic field of class number one,
$E$ is an elliptic curve over $\Q$ with CM by $\OK$,
$\lambda$ is a prime element of $\OK=\End(E)$, 
$P \in E(K) - \lambda E(K)$,
$Q \in \lambda^{-1}(P)$,  and
$L= K(E[\lambda],Q)$.
If the extension $L/K$ is abelian, then the pair 
 $(K,\lambda)$ is one of the following six possibilities: 
\begin{enumerate}[\rm (a)]
\item $K = \Q(\sqrt{-1}) \quad \text{and} \quad \lambda = 1+i$,
\item $K = \Q(\sqrt{-2}) \quad \text{and} \quad  \lambda = \sqrt{-2}$,
\item $K = \Q(\sqrt{-3}) \quad \text{and} \quad  \lambda = 2$ or $\sqrt{-3}$, or
\item $K = \Q(\sqrt{-7}) \quad \text{and} \quad  \lambda = \frac{1 \pm \sqrt{-7}}{2}$.
\end{enumerate}
If we restrict to $\lambda \neq \bar{\lambda}$, then 
only case {\rm (d)} occurs.
\end{theorem}

\begin{proof}
Let $F=K(E[\lambda])$.
Since $P \not\in \lambda E(K)$, we have $L\neq F$.
Since $L/K$ is abelian, $F=K$ by Lemma \ref{abelianeither1}, i.e.,
$E[\lambda] \subseteq E(K)$, so we can apply Theorem \ref{abeliansplit2}.

Using the well known list of imaginary quadratic fields $\Q(\sqrt{-d})$ 
of class number one (sequence A014602 in \cite{oeis}),
it is easy to check that the only time $2$ splits is when $d=7$, and
the only ones where $2$ or 3 ramify
are when $d=1, 2$, or $3$.
The desired result now follows from Theorem \ref{abeliansplit2}.
\end{proof}

\begin{theorem}
\label{classno2examples}
Suppose $K$ is an imaginary quadratic field of class number two
and Hilbert class field $H$,
$E$ is an elliptic curve over $\Q(j(E)) \subset H = K(j(E))$ with CM by $\OK$,
$\lambda$ is a prime ideal of $\OK=\End(E)$, 
$P \in E(H) - \lambda E(H)$,
$\hat{\varphi}: E \to E/E[\bar{\lambda}]$ is the natural isogeny,  
$\varphi$ is the dual isogeny,
$Q \in \varphi^{-1}(P)$,  and
$L= H(E[\lambda],Q)$.
If the extension $L/H$ is abelian, then 
$K = \Q(\sqrt{-d})$ with $d$ in the set $\{5,6,10,13,15,22,37,51,58,123, 267\}$.
The only case where $\lambda \neq \bar{\lambda}$ is when
$K = \Q(\sqrt{-15})$ with $\lambda = (2,\frac{1 + \sqrt{-15}}{2})$
or $(2,\frac{1 - \sqrt{-15}}{2})$.
\end{theorem}

\begin{proof}
As in the previous proof,
since $L/H$ is abelian, we have 
$E[\lambda] \subseteq E(H)$ and we can apply Theorem \ref{abeliansplit2}.
The prime $2$ ramifies in $\Q(\sqrt{-d})$ (with $d$ squarefree) 
if and only if $d\equiv 1,2 \bmod 4$
and it splits if and only if $d\equiv 7 \bmod 8$.
The prime $3$ ramifies in $\Q(\sqrt{-d})$ if and only if $3\mid d$.
Now apply Theorem \ref{abeliansplit2} and
use the well known list of imaginary quadratic fields $\Q(\sqrt{-d})$ 
of class number two (sequence A014603 in \cite{oeis}).
\end{proof}

\begin{remark}
In (b) and (c) of Theorem \ref{abeliansplit2} we have
$\lambda = \overline{\lambda}$.
In general, if
$\lambda = \overline{\lambda}$ for all prime ideals $\lambda$
that divide $\prod \alpha_i$, then there is no need
to use elliptic curve primality tests, since classical primality tests apply.
Thus, the case of Theorem \ref{abeliansplit2} of primary interest to us 
is when 
$p=2$ and $2$ splits in $K/\Q$.
We note that  $2$ splits in $K$ if and only if 
$K = \Q(\sqrt{-d})$ with $d \equiv 7\pmod{8}$.
\end{remark}

\begin{remark}
We leave as an open problem finding a systematic way,
 when $L/M$ is {\em not} abelian, to determine the
$k$'s for which the prime ideal $\frakp_k$ 
splits completely in $F$ but not in $L$.
Doing so would permit more choices of the $\alpha_i$ 
of \S\ref{assmptnsect}.
The case of interest when $L/M$ is abelian,
namely the case where $2$ splits in $K$ and $\lambda\bar{\lambda} = 2$,
essentially reduces to taking each $\alpha_i$ to be $\delta$ or
$\bar{\delta}$ where $\delta$ generates $\lambda^f$ with $f$ the order
of $\lambda$ in the ideal class group of $K$.

\end{remark}

\begin{remark}
Even when $L/M$ is abelian, we do not know how to {\bf systematically}
find ideals $\frakp_k$ such that $N_{\M/K}(\frakp_k)$ is generated by
an element $\pi_k \in\O_K$ for which $\pi_k -1$ is highly factored
(of the form $\gamma\prod \alpha_i^{k_i}$ with the only primes
$\lambda$ of $K$ dividing $\prod\alpha_i$ being as in Theorem \ref{abeliansplit2}).
This is a serious impediment to using the techniques of this paper
to create primality tests when the class number of $K$ is greater than two.
For class number two, in \S\ref{example-15} below it was fortunate that
$\alpha = \frac{1 + \sqrt{-15}}{2}$ satisfies $\alpha=\beta^2$ with
$\beta\in\O_H$ such that when $k$ is odd, 
$\frakp_k = (1+2\beta^k)$ satisfies 
$N_{H/K}(\frakp_k) = 1-4\alpha^k$.
\end{remark}

\begin{remark}
We also leave open the question of generalizing these techniques to
obtain explicit primality tests with 
higher dimensional CM abelian varieties, such as for example
the Jacobian of the hyperelliptic curve $y^2 = x^5 - 1$.
\end{remark}

\section{A primality test using an elliptic curve with CM by $\Q(\sqrt{-2})$}
\label{example-2}

Although the numbers $F_k$ that arise when $M=\Q(\sqrt{-2})$ can be addressed more efficiently by
classical (non-elliptic curve) methods, we state the
result briefly, so that the literature will include
elliptic curve primality tests corresponding
to all the cases in Theorem \ref{fewexamples}.
See \cite{AlexThesis} for details.

Let $K = \mathbb{Q}(\sqrt{-2})$ and
$$E \, : \, y^2 = x^3 -78030x -7428456.$$
Then $E$ is an elliptic curve with CM by $\OK$. The point
$$P = (125\sqrt{-2} - 604, -9190\sqrt{-2} -6700) \in E(K)$$
has infinite order.  
We take $\alpha_1 = \sqrt{-2}$ and $\gamma = 3$,
so
$\pi_k = 1 + 3(\sqrt{-2})^k$, in order to test the primality of 
$F_k = N_{K / \mathbb{Q}}(\pi_k)$.
When $k$ is odd, $F_k = 1+9\cdot 2^k$.
When $k$ is even, then $\pi_k\in\Z$, so $F_k = \pi_k^2$ is composite.
Here, $L_k= 2^{(k+1)/2}$ if $k$ is odd.

\begin{proposition}
\label{Frobmin2}
If $k>1$ is an integer congruent to $1\bmod 4$ and $\pi_k$ is prime, then the Frobenius endomorphism of $E$ over $\OK/(\pi_k)$ is $\pi_k$.
\end{proposition}

\begin{theorem}
\label{mainmin2}
Suppose $k > 1$ is an integer congruent to $1\bmod 8$.
The following are equivalent:
\begin{enumerate}[\rm (a)]
\item
$\pi_k$ is prime;
\item 
$(\sqrt{-2})^{k}(3P)\equiv O_E$ mod $\pi_k$ and 
$(\sqrt{-2})^{k-1}(3P)$ is strongly nonzero mod $\pi_k$;
\item 
$2^{(k+1)/2}(3P)\equiv O_E$ mod $\pi_k$ and 
$2^{(k-1)/2}(3P)$ is strongly nonzero mod $\pi_k$.
\end{enumerate}
\end{theorem}

\section{A primality test using an elliptic curve with CM by $\Q(\sqrt{-15})$}
\label{example-15}

We work over the following fields:
\begin{center}
\begin{tikzpicture}[node distance=2.3cm, auto]
\node                            (H) {$H = \Q(\sqrt{-3},\sqrt{5})$};
\node [below of = H] (K) {$K = \Q(\sqrt{-15})$};
\node [left of = K]        (K_1) {$K_1 = \Q(\sqrt{-3})$};
\node [right of = K]      (K_2) {$K_2 = \Q(\sqrt{5})$};
\node [below of = K]  (Q) {$\Q$};
\draw (H) -- (K) -- (Q) -- (K_2) -- (H) -- (K_1) -- (Q);
\end{tikzpicture}
\end{center}
The field $H$ has class number one and is the Hilbert class field of $K$. We use the elliptic curve
\begin{equation*}
E: y^2 = x^3 + a_4x + a_6,
\end{equation*}
where
\begin{align*}
a_4 &:= -3234(16195646845 - 7242913457\sqrt{5}), \\
a_6 &:=  14^4(5395199151946361 - 2412806411180256\sqrt{5}).
\end{align*}
Let 
$$
P := (0, -14^2(51938421 - 23227568\sqrt{5})) \in E(K_2) \subseteq E(H),
$$
and let
$$
\alpha := \frac{1+\sqrt{-15}}{2}, \qquad \pi_k := 1 - 4\alpha^k \in \OK.
$$
Then $\alpha\bar{\alpha}=4$.  We note that $(\alpha)=\lambda^2$,
where $\lambda$ is the prime $\OK$-ideal
$$
\lambda:=(2,\alpha),
$$
with $N_{K/\Q}(\lambda)=\lambda\bar{\lambda}=2$.

The numbers we will test for primality are those in the sequence
\begin{equation*}
F_k := N_{K/\Q}(\pi_k) = 1 - 4(\alpha^k + \bar{\alpha}^k) + 4^{k+2} \in \Z,
\end{equation*}
where $k$ lies in the set
\begin{align*}
S := \{k \in \N: k \equiv &\ 9, 19, 39, 45, 59, 63, 67, 85, 105, 123, 129, 133, 159\\
&\ 169, 173, 181, 183, 221, 223, 225, 229 \bmod{240}\}.
\end{align*}

We also define
$$
\beta := \frac{\sqrt{5}+\sqrt{-3}}{2} \in \OH, \qquad p_k := 1 + 2\beta^k \in \OH.
$$ 
We have $\beta^2 = \alpha$, $\beta\bar{\beta}=2$, and note that
$\beta$ and $\bar{\beta}$ are generators of the (principal) prime ideals of $\OH$ above $2$.
When $k$ is odd (in particular, for $k\in S$), we have
$$
N_{H/K}(p_k) = (1 + 2\beta^k)(1 - 2{\beta}^k) = 1 - 4\alpha^k  = \pi_k,
$$ 
and
$N_{H/\Q}(p_k) = \pi_k\bar{\pi}_k = F_k.$  We also define
\begin{align*}
\pi_{K_1,k} &:= N_{H/K_1}(p_k) = 1 + 2(\beta^k + (-\bar{\beta})^k) + (-1)^k2^{k+2},\\
\pi_{K_2,k} &:= N_{H/K_2}(p_k) = 1 + 2(\beta^k + \bar{\beta}^k) + 2^{k+2}.
\end{align*}

With this setup, Theorem~\ref{mainthm1} yields the following primality criterion for $F_k$.
   
 \begin{theorem}
\label{mainthm}
Suppose $k \in S$. The following are equivalent:
\begin{enumerate}[\rm (a)]
\item $F_k$ is prime.
\item $4\alpha^kP \equiv O_E \bmod{p_k}$ and $8\alpha^{k-1}P$ is strongly nonzero mod $(p_k)$.
\item $2^{2k+2}P \equiv  O_E \bmod{p_k}$ and $2^{2k+1}P$ is strongly nonzero mod $(p_k)$.
\item $2^{2k+1}P \equiv (7(377709\sqrt{5} - 844583),0) \bmod{p_k}$.
\end{enumerate}
\end{theorem}

We will prove Theorem~\ref{mainthm} in \S \ref{mainthmproof}.
\smallskip

In order to turn Theorem~\ref{mainthm} into an efficient algorithm, rather than working with the reduction of the $E$ modulo $(p_k)$, we prefer to work with the reduction of a curve $E_d$ modulo $F_k$, where $E_d$ is defined so that the reduction of $E$ modulo $(p_k)$ is isomorphic to the reduction of $E$ modulo $F_k$ in the case that $F_k$ is prime; the parameter $d\in \Z$ will be chosen so that its reduction in $\Z/F_k\Z$ is a square root of~5.  We thus define
\begin{align*}
a_{4,d} &:=  -3234\, (16195646845 - 7242913457d),\\
a_{6,d} &:= 38416\, (5395199151946361 - 2412806411180256d),
\end{align*}
and let $E_d \colon y^2 = x^3 + a_{4,d}x + a_{6,d}$ and 
$P_d :=  (0,-10179930516 + 4552603328d)$.  
We note that if $d^2 \equiv 5 \bmod{F_k}$, then $P_d \in E_d(\Z/F_k\Z)$.

We now give a primality criterion for $F_k$ in terms of $E_d$ and $P_d$.

\begin{theorem}
\label{equivdrew}  
Suppose $k \in S$.  The following are equivalent:
\begin{enumerate}[\rm (a)]
\item $F_k$ is prime.
\item 
There exist $d\in\Z$ and relatively prime $x,y,z\in\Z$ such that $d^2 \equiv 5 \bmod{F_k}$
and $[x:y:z]\equiv 2^{2k+1}P_d\bmod F_k$ with
$$
\gcd(z,F_k) = 1 \qquad \text{and} \qquad y \equiv 0 \bmod{F_k}.
$$
\end{enumerate}
\end{theorem}

The proof of Theorem~\ref{equivdrew} is given in \S \ref{equivdrewproof}.
It yields the following algorithm.

\begin{algorithm}
\label{sutherlandalg}
For $k \in S$, determine the primality of $F_k$ as follows:
\begin{enumerate}[\rm 1.]
\item If $5^{(F_k - 1)/4} \not\equiv \pm 1 \bmod{F_k}$, output ``$F_k$ is composite'' and terminate.
\item Compute $e := (F_k - 5)/8\in\Z$.
\item If $5^{(F_k -1)/4} \equiv 1 \bmod{F_k}$, compute $d := 5^{e + 1}$ mod $F_k$.\\
Otherwise, compute $d := 2^{2e + 1}5^{e + 1}$ mod $F_k$.
\item If $d^2 \not\equiv 5 \bmod{F_k}$, output ``$F_k$ is composite'' and terminate.
\item Let $\bar{P}$ be the reduction of $P_d$ modulo $F_k$.
\item Compute $\bar{Q} = 2^{2k+1}\bar{P} \in E_d(\Z / F_k \Z)$ as $\bar{Q} = [\bar{x} : \bar{y} : \bar{z}]$, and lift  $\bar{x},\bar{y},\bar{z}$ to 
relatively prime $x,y,z\in \Z$.
\item If $y \not\equiv 0 \bmod{F_k}$ or $\gcd(z,F_k)\neq 1$, output ``$F_k$ is composite'' and terminate.\\
Otherwise, output ``$F_k$ is prime."
\end{enumerate}
\end{algorithm}

\begin{remark}
The elliptic curve group operations used to compute $\bar{Q}$ in step 6 uses formulas for the group law in projective coordinates that are well defined over the ring $\Z/F_k\Z$, whether or not $F_k$ is prime.
In the case that $F_k$ is prime, the point $\bar{P}\in E_d(\F_{F_k})$ constitutes a Pomerance proof of the primality of $F_k$.
\end{remark}

The correctness of Algorithm~\ref{sutherlandalg} is proved in \S \ref{algcorrectness}; here we note that its complexity is quasi-quadratic in $k=O(\log F_k)$.

\begin{theorem}\label{thm:complexity}
The time complexity of Algorithm~\ref{sutherlandalg} is $O(k^2\log k\log\log k)$.
\end{theorem}
\begin{proof}
The complexity of exponentiation in $\Z/F_k\Z$ and scalar multiplication in $E_d(\Z/F_k\Z)$ (using projective coordinates) is $O(n\MM(n))$, where $n$ is the number of bits in the exponent/scalar and $\MM(n)$ is the cost of multiplying two $n$-bit integers.
The binary representations of the integers $F_k$ and $e$ both consist of $O(k)$ bits, and applying the Sch\"onhage-Strassen \cite{schonhagestrassen} bound $\MM(n)=O(n\log n\log\log n)$ yields an $O(k^2\log k\log\log k)$ bound for steps 1--6.
This dominates the time to compute the gcd in step 7 using the Euclidean algorithm.
\end{proof}

The bound in Theorem~\ref{thm:complexity} can be slightly improved by using F\"urer's algorithm~\cite{Furer07} for integer multiplication (see \cite{HvdHL14} for further refinements).
In order to simplify the proofs that follow, and for the purposes of efficiently computing $F_k$, we note the following recurrence relation.

\begin{proposition}
\label{recursiveFk}
We can define $F_k$ recursively as follows:
\begin{equation*}
F_0 = 9, \qquad F_1 = 61, \qquad F_k = F_{k-1} - 4F_{k-2} + 4^{k+2} +
4\quad(k\ge 2).
\end{equation*}
\end{proposition}

\begin{proof}
Recall that $F_k=N_{K/\Q}(1-4\alpha^k) =
1-4(\alpha^k+\bar\alpha^k)+4^{k+2}$, where $\alpha=\frac{1+\sqrt{-15}}{2}$.
Thus $F_0=N_{K/\Q}(-3)=9$ and $F_1=1-4(\alpha+\bar\alpha)+64= 61$.
For $k\ge 2$ we have
\begin{align*}
F_k &=  1-4(\alpha^k+\bar\alpha^k)+4^{k+2}\\
&= 1-4(\alpha^{k-2}\alpha^2+\bar\alpha^{k-2}\bar\alpha^2)+4^{k+2}\\
&= 1-4(\alpha^{k-2}(\alpha-4)+\bar\alpha^{k-2}(\bar\alpha-4))+4^{k+2}\\
&= 1-4(\alpha^{k-1}+\bar\alpha^{k-1}) +
4^2(\alpha^{k-2}+\bar\alpha^{k-2})+4^{k+2}\\
&= F_{k-1}-4^{k+1} +4(1-F_{k-2}+4^k)+4^{k+2}\\
&= F_{k-1}-4F_{k-2}+4^{k+2}+4. \qedhere
\end{align*}
\end{proof}

\subsection{Preparation}
\label{addlnotationsect}
In order to prove Theorems~\ref{mainthm} and~\ref{equivdrew} we need to determine the values of $k$ to which we can apply Theorems~\ref{mainthm1} and~\ref{mainthm2}.
For this purpose, we define the following three sets of positive integers:
\begin{align*}
T_1 &:= \{k \in \N: k \equiv 3, 9, 13, 19, 21, 27, 31, 39, 45, 47, 49, 53, 59, 61, 63, 65, 67, 81, 85, \\
&\hspace{30pt}91, 101, 103, 105, 109, 113, 117 \bmod{120}\},\\
T_2 &:= \{k \in \N: k \equiv 1, 5, 7, 9, 17, 19, 23, 27, 31, 35, 39, 41, 43, 45, 51, 55, 59, 63, 67, \\
&\hspace{30pt}69, 71, 81, 83, 85, 89, 95, 97, 99, 105, 119, 123, 129, 131, 133, 137, 
141, 143,  \\
&\hspace{30pt}145, 149, 157, 159, 161, 169, 173, 181, 183, 191, 193, 195, 197, 199, 201,\\
&\hspace{30pt}209, 211, 213, 215, 221, 223, 225, 227, 229, 235, 237, 239 \bmod{240}\},\\
T_3 &:= \{ k \in \N  : k\equiv 27,31,81,141,201,211,237 \bmod{240}\}.
\end{align*}

We show in \S\ref{secT1} that $T_1$ is precisely the set of $k$ for which $\pi_k$ is the Frobenius endomorphism of the reduction of $E$ modulo $p_k$, whenever $\pi_k$ is prime, and we show in \S\ref{secT2} that $T_2$ is precisely the set of $k$ for which the reduction of $P$ modulo~$p_k$ does not lie in $\lambda(E(\OH/(p_k)))$, whenever $\pi_k$ is prime.
These are precisely the assumptions (i) and (ii) of \S\ref{mainresults} required by Theorems~\ref{mainthm1} and~\ref{mainthm2}.

Theorems~\ref{mainthm} and~\ref{equivdrew} actually apply to all $k$ in $T_1\cap T_2$, but for the sake of efficiency we can rule out $k\in T_3$, since $F_k$ is necessarily composite for all such $k$, as proved below.
This yields the set $S = (T_1 \cap T_2) - T_3$ defined above.
\smallskip

\begin{lemma}~\\\vspace{-16pt}
\label{recursiveprops}
\begin{enumerate}[\rm (a)]
\item 
$F_k$ is divisible by $3$ if and only if $k$ is even.
\item
$F_k$ is divisible by $5$ if and only if $k \equiv 2 \bmod{4}$.
\item
$F_k$ is divisible by $7$ if and only if $k \equiv 16 \bmod{24}$.
\item
$F_k$ is divisible by $11$ if and only if $k \equiv 48 \bmod{60}$.
\item
$F_k$ is divisible by $31$ if and only if $k \equiv 6$ or $12 \bmod{15}$.
\item
$F_k$ is divisible by $61$ if and only if $k \equiv 1 \bmod{30}$.
\item
If $k \not\equiv 2 \bmod{4}$, then $F_k \equiv \pm 1\bmod{5}$.
\item If $k\ge 1$, then $F_k \equiv 5 \bmod{8}$.
\item
If $k \in T_3$, then $F_k$ is composite. 
\end{enumerate}
\end{lemma}

\begin{proof}
Parts (a,b,g,h) follow from Proposition \ref{recursiveFk} by induction.  

For (c) and (d), since
$\alpha^{24} \equiv 1$ mod ${7}$ and  $\alpha^{60} \equiv 1$ mod ${11}$,
$\pi_k$ mod $7$ (resp., $11$) 
depends only on the congruence class of $k$ mod $24$ (resp., $60$). 
Since $7$ and $11$ are inert in $K/\Q$, (c,d) follow by 
calculating 
$\pi_k$ mod ${7}$ for $1\le k \le 24$, and $\pi_k$ mod ${11}$ for $1\le k \le 60$. 

For (e),
let $\ell_1 = 4 + \sqrt{-15}$ and $\ell_2 = 4 - \sqrt{-15}$.  We have $\ell_1\ell_2=31$, thus $31 | F_k$ if and only if  $\ell_1 | \pi_k$ or $\ell_2 | \pi_k$ in $\OK$.
Using Sage \cite{sage}, we find that 
$\alpha^{15} \equiv 1 \bmod{\ell_i}$ for $i = 1,2$.
Thus $\pi_k \bmod{\ell_i}$ depends only on $k \bmod{15}$.  
We then compute
$\pi_k \bmod{\ell_i}$ for $k = 0,\ldots,14$ and $i = 1,2$,
and find that
$\ell_1 | \pi_k$ if and only if $k=6$ and $\ell_2 | \pi_k$ if and only if $k=12$.

For (f),
let $\ell_1 = -1 - 2\sqrt{-15}$ and $\ell_2 = -1 + 2\sqrt{-15}$.  We have $\ell_1\ell_2=61$
and find that $\alpha^{30} \equiv 1 \bmod{\ell_i}$ for $i=1,2$, so
$\pi_k \bmod{\ell_i}$ depends only on $k \bmod{30}$.
We then compute $\pi_{k} \bmod{\ell_i}$ for $k = 0,\ldots,29$ and $i = 1, 2$,
and find that $\ell_2 | \pi_k$ if and only if $k = 1$ and $\ell_1 \nmid \pi_k$.   

Part (i) follows from parts (a)--(f).
\end{proof}

We also note some additional lemmas and definitions that will be used below.

\begin{lemma}
\label{equiv}
If $k$ is odd then the following are equivalent:
\begin{enumerate}[\rm (a)]
\item $F_k$ is a rational prime,
\item  $\pi_k$ is a prime of $\O_K$,
\item $p_k$ is a prime of $\O_H$.
\end{enumerate}
\end{lemma}

\begin{proof}
If $F_k$ is prime, then $\pi_k$ is prime,
since $F_k = N_{K/\Q}(\pi_k)$.
If $\pi_k$ is prime, then $p_k$ is prime, since 
$\pi_k = N_{H/K}(p_k)$.
If $p_k$ is prime, then $F_k=p^f$ for some rational prime $p$
with $f\in \{1,2,4\}$, so $F_k$ is a prime or a square.
But if $k > 0$,
then $F_k \equiv 5 \bmod{8}$ is not a square modulo $8$, so $F_k$ is not a square.
\end{proof}

\begin{lemma}
\label{goodreduction}
If $k \in S$, and  $p \in \Z$ is a prime divisor of $F_k$, then:
\begin{enumerate}[\rm (a)]
\item  $E$ has good reduction modulo every prime ideal of $\OH$ above $p$, and
\item $E_d$ has good reduction modulo $p$ for every $d \in \Z$ with $d^2 \equiv 5 \bmod{F_k}$.
\end{enumerate}
\end{lemma}

\begin{proof}
The absolute norm of the discriminant of $E$ is $2^{36}3^67^{12}11^6$.
Since $k\in S$ and $p\mid F_k$, by Lemma \ref{recursiveprops}
we have $p \neq 2,3,7,11$, giving (a).

If $d^2 \equiv 5 \bmod{F_k}$,
then  
$\norm_{K_2 / \Q}(\text{disc}(E)) \equiv \text{disc}(E_d)\text{disc}(E_{-d})$ mod ${p}$,
and (b) follows from (a).
\end{proof}

\begin{definition}
If  $\p$ is a prime ideal in $\OH$, and $a, b \in H_\p$, then the Hilbert symbol is defined by
$$\left(\frac{a,b}{\p}\right) =  \left\{
	\begin{array}{rl}
		1 & \text{if $z^2 = ax^2 + by^2$ has a solution in $H_{\p}^3 - \{(0,0,0)\}$},\\ 
		-1 & \text{otherwise.}
	\end{array}
\right.$$
\end{definition}

\begin{lemma}
\label{hilbertsquare}
If $\p$ is a prime ideal in $\OH$ and $a,b \in H_\p$, with $b$ square, then $\left(\frac{a,b}{\p}\right) = 1$.
\end{lemma}

\begin{proof}
Take $\zeta \in H_\p$ such that $\zeta^2 = b$. Then $(x,y,z)=(0, 1, \zeta)$ is a solution of $z^2 = ax^2 + by^2$.
\end{proof}

\subsection{The set $T_1$}\label{secT1}
In this section we prove that the set $T_1$ defined in \S\ref{addlnotationsect} is precisely the set of $k$ for which
$\pi_k$ is the Frobenius endomorphism of the reduction of $E$ modulo $p_k$,
whenever $\pi_k$ is prime.

\begin{proposition}
\label{frobenius}
\label{structure}
Suppose $k\in T_1$ and $p_k$ is prime in $\OH$.
Then:
\begin{enumerate}[\rm (a)]
\item The Frobenius endomorphism of $E$ over $\OH/(p_k)$ is $\pi_k$.
\item $E(\OH/ (p_k)) \simeq \OK/(4\alpha^k)$ as $\OK$-modules.
\item $E(\OH/ (p_k)) \simeq \Z/4\Z \times \Z/4^{k+1}\Z$ as groups.
\end{enumerate}
\end{proposition}

\begin{proof}
Since $p_k$ is prime in $\OH$, $k$ is odd by Lemma \ref{recursiveprops}(a). 
Let $\eta=\frac{-13+5\sqrt{5}}{2}$.
By Theorem 5.3 of \cite{rubinsilverberg}, the Frobenius endomorphism of $E$ over $\OH/(p_k)$ is
$$
\left(\frac{2\cdot7\sqrt{-3}\eta}{(p_k)}\right)\pi_k
	= \left(\frac{2}{F_k}\right)\left(\frac{7}{F_k}\right)\left(\frac{\sqrt{-3}}{\pi_{F,k}}\right)\left(\frac{\eta}{\pi_{K_2,k}}\right)\pi_k,
$$
where the coefficients of $\pi_k$ are all generalized Legendre symbols.

By Lemma \ref{recursiveprops} (h), $F_k \equiv 5 \bmod{8}$. 
Hence, $\left(\frac{2}{F_k}\right) = -1$ and $\left(\frac{7}{F_k}\right) = \left(\frac{F_k}{7}\right)$. 
Since $\alpha$ and $\bar{\alpha}$ both have order 24 in $(\OK/(7))^{\times}$, 
we find that
\begin{equation*}
\left(\frac{7}{F_k}\right) = \left(\frac{F_k}{7}\right) = \left\{
	\begin{array}{rl}
		1 & \text{if } k \equiv 9, 15, 17, 23 \bmod{24},\\
		-1 & \text{for all other odd $k$.}
	\end{array}
\right.
\end{equation*}

By Theorem 8.15 of \cite{Lem}, we have 
$\left(\frac{\sqrt{-3}}{\pi_{F,k}}\right)\left(\frac{\pi_{F,k}}{\sqrt{-3}}\right) = -1$,  
and one can show that $\pi_{F,k} \equiv 2 \bmod{\sqrt{-3}}$ for all odd $k$.
Thus for all odd $k$ we have 
\begin{equation*}
\left(\frac{\sqrt{-3}}{\pi_{F,k}}\right) = -\left(\frac{\pi_{F,k}}{\sqrt{-3}}\right) = -\left(\frac{2}{\sqrt{-3}}\right) = -\left(\frac{2}{3}\right) = 1.
\end{equation*}

By Theorem 12.17 of \cite{lemmermeyer}, 
$\left(\frac{\eta}{\pi_{K_2,k}}\right)\left(\frac{\pi_{K_2,k}}{\eta}\right) = -1$. 
Using Sage, since  $\beta$ and $\overline{\beta}$ have order 120 in $(\OH/\mathfrak{P}_{11})^{\times}$ (where $\mathfrak{P}_{11}$ is a prime above 11 in $\OH$),  
and $\left(\frac{\eta}{\pi_{K_2,k}}\right)=-\left(\frac{\pi_{K_2,k}}{\eta}\right)$, 
we find that
$$
\left(\frac{\eta}{\pi_{K_2,k}}\right) =
\begin{cases}
1 & \text{if } k \equiv 3, 13, 15, 17, 19, 21, 23, 27, 31, 33, 41, 45, 49, 53, 57, 59, 61, \\
   &67, 71, 85, 87, 89, 91, 95, 101, 103, 109, 111, 117, 119
			 \bmod{120},\\
-1 & \text{for all other odd $k$.}
\end{cases}
$$
Part (a) now follows.
By (a), as $\O_K$-modules we have 
$$
{E}(\OH/(p_k)) \simeq \ker(\pi_k - 1)  =  \ker(4\alpha^k) 
\simeq  \OK/(4\alpha^k),        
$$
which proves (b).
Part (c) follows from (b), since $\alpha\overline{\alpha}=4$.
\end{proof}

\subsection{The set $T_2$}\label{secT2}
In this section we prove that the set $T_2$  defined in \S\ref{addlnotationsect} is precisely
the set of $k$ for which the reduction of $P$ modulo $(p_k)$ does not lie in
$\lambda(E(\OH/(p_k)))$, whenever $\pi_k$ is prime (see Proposition \ref{image}).
We first use Sage to compute the action of the endomorphism $\alpha$ on $E$, which is recorded in the following lemma.

\begin{lemma}
\label{alphaaction}
If $Q = (x,y) \in E$, then
the $x$-coordinate of $\alpha Q$ is ${f(x)}/{g(x)}$, where
\begin{align*}
f(x) =\ &(299537289 + 133957148\sqrt{5})x^4\\
&+ ({646275} - {96341}\sqrt{-3} + {289023}\sqrt{5} - {43085}\sqrt{-15})x^3/2\\
& + (257691 + 185465\sqrt{-3} - 119511\sqrt{5} - 75313\sqrt{-15})x^2\\
	&+ (-1639595729268 - 1831800977776\sqrt{-3}\\
	&\quad\ + 733249501264\sqrt{5}+ 819206301452\sqrt{-15})x\\
	&+ 3260424679620398892 + 5199743168890017300\sqrt{-3}\\
	  &- 1458106243829837028\sqrt{5} - 2325395838235649676\sqrt{-15},
\end{align*}
and
\begin{align*}
g(x) =\ &(-2096761023 + 669785740\sqrt{-3} - 937700036\sqrt{5} + 299537289\sqrt{-15})x^3/2\\
	&+ (-{1938825} + {1059751}\sqrt{-3} - {867069}\sqrt{5} + {473935}\sqrt{-15})x^2/2\\
	&+ (-337071 - 275233\sqrt{-3} + 140091\sqrt{5} + 133721\sqrt{-15})x\\
	&+ 547023393084 + 809830063056\sqrt{-3}\\
	&- 244636298480\sqrt{5} - 362167014244\sqrt{-15}.
\end{align*}
\end{lemma}

Let $p_{K_1,7} := \frac{-1 + 3\sqrt{-3}}{2}$ (a prime above 7 in $K_1$) and 
$p_{K_2,11} := \frac{1 + 3\sqrt{5}}{2}$ (a prime above 11 in $K_2$).
Recall that $\beta = ({\sqrt{5}+\sqrt{-3}})/{2}$, and let
$$
\delta_P := p_{K_1,7} \cdot p_{K_2,11} \cdot \beta \in H.
$$

\begin{lemma}
\label{lambdaxQ}
Let  $E'=E/E[\bar{\lambda}]$, let 
$\hat{\varphi}\colon E \to E'$ denote the natural isogeny, and
let $\varphi\colon E' \to E$ denote the dual isogeny.
Choose $R\in E(\bar{\Q})$ so that $\varphi(R)=P$.
Then:
$$
H(R)  = H(x(R)) = H(\sqrt{\delta_P}),
$$ 
where $x(R)$ is the $x$-coordinate of $R$, and
$[H(R):H]=2$.
\end{lemma}

\begin{proof}
We first calculate that the unique point $Q \in E[\bar{\lambda}] -\{ 0 \}$ is $Q := 
(A,0)$, where
$$
A:=7(11933\sqrt{-15} - 377709\sqrt{5} -26683\sqrt{-3} + 844583)/2.
$$
Writing $E$ as $y^2=f(x)$ and
letting $E_1$ be the elliptic curve $y^2=f(x+A)$, 
the map  $\psi(x,y)=(x+A,y)$ is an isomorphism from $E_1$ to $E$ that
takes $(0,0)$ to $Q$.
Example 4.5 on p.~70 of \cite{silverman09} tells us how to explicitly compute
$2$-isogenies. In particular, for
the natural isogeny $E_1 \to E_1/E_1[\bar{\lambda}] = E'$,
it gives a formula for the dual isogeny
$\hat{\phi} : E'  \to E_1$.
Since $\varphi = \psi\circ\hat{\phi}$, we find that the $x$-coordinate $x(R)$
is a root of an irreducible quadratic $h(x) \in H[x]$ 
whose discriminant is $u^2\delta_P$ where 
$u = 28(\sqrt{-3} + 1)(-40895\sqrt{5} + 91444).$  
Thus, 
$H(x(R)) = H(\sqrt{\text{disc}(h)}) = H(\sqrt{\delta_P})$, and this field
has degree 2 over $H$.
Further, one may check that $y(R) \in H(x(R))$, so $H(R) = H(x(R)) = H(\sqrt{\delta_P})$.
\end{proof}

\begin{remark}
\label{7and11}
Using Sage, we determined
that if  $\mathfrak{p}$ is a prime ideal of $\OH$  above~7 (resp., 11),
then both $\beta$ and $\bar{\beta}$ have order 48 (resp., 120) 
in $(\OH/\mathfrak{p})^{\times}$.
\end{remark}

\begin{proposition}
\label{legendre-1}
If $p_k$ is prime in $\OH$ then the following hold:
\begin{enumerate}[\rm (a)]
\item
 $P \not\in \lambda(E(\OH/(p_k)))$ if and only if $(\frac{\delta_P}{p_k})= -1$;
\item 
$\left(\frac{p_{K_1,7}}{\pi_{K_1,k}}\right) = -\left(\frac{\pi_{K_1,k}}{p_{K_1,7}}\right) =
\begin{cases}
			\phantom{-}1 & \text{if  $k \equiv 1, 3, 5, 7, 9, 13, 15, 21, 33, 35, 39, 43
				  \bmod{48}$},\\
			-1 & \text{for all other odd $k$;}
\end{cases}
$
\item 
$\left(\frac{p_{K_2,11}}{\pi_{K_2,k}}\right) = \left(\frac{\pi_{K_2,k}}{p_{K_2,11}}\right) =$ \\
\begin{eqnarray*}
\qquad
	 \left\{
	\begin{array}{rl}
		1 & \text{if } k \equiv 1, 5, 7, 9, 11, 25, 29, 35, 37, 39, 43, 47, 51,
					  55, 63, 65, 69,  \\
			&	73, 75, 77, 79, 81, 83, 93, 97,
					  99, 105,107, 113, 115 \bmod{120},\\
		-1 & \text{for all other odd $k$;}
	\end{array}
\right.
\end{eqnarray*}
\item
$(\frac{\beta}{p_k}) = -1.$
\end{enumerate}
\end{proposition}

\begin{proof}
By Theorem \ref{splittingiff} (and Lemma \ref{goodreduction}(a)), 
$\tilde{P} \in \lambda(E(\OH/(p_k)))$
if and only if $(p_k)$ splits in the quadratic extension 
$H(\varphi^{-1}(P))=H(\sqrt{\delta_P})$.
Part (a) now follows.

Theorem 8.15 of \cite{Lem} implies
$\left(\frac{p_{K_1,7}}{\pi_{K_1,k}}\right)\left(\frac{\pi_{K_1,k}}{p_{K_1,7}}\right) = -1$,
and (b) then follows from Remark \ref{7and11}.

Theorem 12.17 of \cite{lemmermeyer} implies
$\left(\frac{p_{K_2,11}}{\pi_{K_2,k}}\right)\left(\frac{\pi_{K_2,k}}{p_{K_2,11}}\right) = 1$,
and (c) then follows from Remark \ref{7and11}. 

Using (a direct generalization of) Theorem III.1 of \cite{serre73}, 
we can show that 
$\left(\frac{\beta,p_k}{\beta}\right) = 1$ and
$\left(\frac{\beta,p_k}{p_k}\right) = 
\left(\frac{\beta}{p_k}\right).$

Let $f(x) = x^2 - p_k$.
Then $f(1) = 1 - p_k = 2\beta^k$ and $f'(1)^2 = 4$.  Thus
$$
|f(1)|_{\beta} = \frac{1}{2^{k+1}} \leq \frac{1}{4} = |f'(1)^2|_{\beta}.
$$
By Hensel's Lemma, $f(x)$ has a root in $H_{\beta}$; equivalently, $p_k$ is a square in $H_{\beta}$.  
By Lemma \ref{hilbertsquare}, 
$
\left(\frac{\beta,p_k}{\beta}\right) = 1.
$

Taking $q$ to run over all primes of $\OH$, and noting all archimedean places of $H$ are complex, 
by the above and the product formula we have
$$
1 = \prod_{q}{\left(\frac{\beta,p_k}{q}\right)} 
  = \left(\frac{\beta,p_k}{p_k}\right)\left(\frac{\beta,p_k}{\bar{\beta}}\right)\left(\frac{\beta,p_k}{\beta}\right) 
  =  \left(\frac{\beta}{p_k}\right)\left(\frac{\beta,p_k}{\bar{\beta}}\right),
$$
and thus
\begin{equation}
\label{betapk}
\left(\frac{\beta}{p_k}\right) = \left(\frac{\beta,p_k}{\bar{\beta}}\right).
\end{equation}

Suppose $\gamma, \delta \in\mathcal{O}_{H_{\bar{\beta}}}^\times$ 
and $\gamma \equiv \delta \bmod{\bar{\beta}^3}$.  
Then $\gamma\delta^{-1} \equiv 1 \bmod{\bar{\beta}^3}$, so 
$\gamma\delta^{-1} = \epsilon^2$ for some $\epsilon \in \mathcal{O}_{H_{\bar{\beta}}}^ \times$
by Proposition XIV.9 of \cite{Serre79}. 
Thus $\gamma\delta = (\delta\epsilon)^2$. 
By Lemma \ref{hilbertsquare} we have
$$
1 = \left(\frac{\beta, (\delta\epsilon)^2}{\bar{\beta}}\right) = \left(\frac{\beta,\gamma\delta}{\bar{\beta}}\right) = \left(\frac{\beta,\gamma}{\bar{\beta}}\right)\cdot\left(\frac{\beta,\delta}{\bar{\beta}}\right),
$$
and therefore
$$
\left(\frac{\beta,\gamma}{\bar{\beta}}\right) = \left(\frac{\beta,\delta}{\bar{\beta}}\right).
$$
Since $\beta \nmid p_k$ we have $p_k\in\mathcal{O}_{H_{\bar{\beta}}}^\times$.  
By the above, $\left(\frac{\beta,p_k}{\bar{\beta}}\right)$ depends only on 
$p_k \bmod{\bar{\beta}^3}$.
Since $\bar{\beta}^3 = \frac{-\sqrt{5} - 3\sqrt{-3}}{2}$,
we have $\beta = \frac{\sqrt{5} + \sqrt{-3}}{2} \equiv \sqrt{-3} \bmod{\bar{\beta}^3}$.  
Since $N_{H/K_1}(\bar{\beta}) = 2$, it follows that $\beta^4 \equiv 1 \bmod{\bar{\beta}^3}$,     
so
\begin{equation}
\label{kmod4}
p_{k+4} = 1 + 2\beta^{k+4} \equiv 1 + 2\beta^k \equiv p_k \bmod{\bar{\beta}^3}.
\end{equation}
By \eqref{betapk} and\eqref{kmod4},   
$\left(\frac{\beta}{p_k}\right)$
depends only on the congruence class of $k \bmod{4}$.  
By Lemma \ref{recursiveprops}(a), $k$ is odd. 
Note that $p_5$ and $p_{15}$ are both primes in $\O_H$. 
We use Sage to compute that $\left(\frac{\beta}{p_5}\right) = \left(\frac{\beta}{p_{15}}\right) = -1$.
Thus, $\left(\frac{\beta}{p_k}\right) = -1$ whenever 
$k \equiv \pm 1 \bmod{4}$, giving (d). 
\end{proof}

\begin{proposition}
\label{image}
If $k\in T_2$ and $p_k$ is prime in $\OH$, then $P \not\in \lambda(E(\OH/ (p_k) ))$.
\end{proposition}
\begin{proof}
This follows directly from Proposition \ref{legendre-1} and the fact that
$$
\left(\frac{\delta_P}{p_k}\right) = \left(\frac{p_{K_1,7}}{\pi_{K_1,k}}\right)\left(\frac{p_{K_2,11}}{\pi_{K_2,k}}\right)\left(\frac{\beta}{p_k}\right).\qedhere
$$
\end{proof}

\subsection{Proof of Theorem \ref{mainthm}}\label{mainthmproof}

We apply Theorem \ref{mainthm1} with $\gamma=-4$ and $\alpha_1=\alpha$
(and Lemma \ref{equiv}).
Suppose $k\in S$. Then $k \ge 9$,
and
$$
F_k \ge F_9 = 1050139 > 16^3 =16\norm_{K/\Q}(\gamma)^2.
$$
By Lemma \ref{goodreduction}, we have $\gcd(\text{disc}(E),F_k)=1$.

By Theorem \ref{mainthm1} and Propositions \ref{frobenius} and \ref{image},
we have $4\alpha^kP \equiv 0 \bmod p_k$, and there is a point in
$\frac{(4\alpha^k)}{\lambda}P$ that is strongly nonzero mod $(p_k)$.
We have 
$
(4\alpha^k) / \lambda =  4\alpha^{k-1}\lambda = (8\alpha^{k-1}, 4\alpha^k).
$
Since $4\alpha^kP \equiv 0$ mod $p_k$, we conclude
that $8\alpha^{k-1}P$ is strongly nonzero mod $(p_k)$, giving (b). 

Since $\alpha\overline{\alpha}=4$ we have $L_k=4^k$.
The equivalence of (a), (b), and (c) now follows from Theorem \ref{mainthm1}.

For (d), suppose $F_k$ is prime.  
Then $p_k$ is prime in $\OH$. 
Since (a) $\Rightarrow$ (b), we have $2^{2k+1}P \not\equiv 0_E \bmod{p_k}$ and $2^{2k+2}P \equiv 0_E \bmod{p_k}$.  
By Proposition \ref{frobenius}(b), $4\alpha^kP \equiv 0_E \bmod{p_k}$.
Since $4=\alpha\bar{\alpha}$ we have
$$
\alpha\cdot 2^{2k+1}P = 2\cdot \bar{\alpha}^{k-1}(4\alpha^kP) \equiv 0_E \bmod{p_k}.
$$
So $2^{2k+1}P \bmod{p_k}$ is a non-trivial point killed by $2$ and $\alpha$.  
Using Sage, we calculate that the only point $Q\in E(\Q(\sqrt{5}))-\{0_E\}$
such that $2Q \equiv \alpha Q \equiv 0_E \bmod{p_k}$ is the point
$Q=(2643963\sqrt{5} - 5912081,0)$.  
This gives (d).  
Conversely, suppose (d) holds.  
Then $2^{2k+1}P$ is strongly non-zero mod $(p_k)$. 
Since the $y$-coordinate of $2^{2k+1}P$ is 0, we have $2^{2k+2}P \equiv 0_E \bmod{p_k}$, giving (b).\qed

\subsection{Proof of Theorem \ref{equivdrew}}\label{equivdrewproof}

Suppose $k\in S$ and $F_k$ is prime.  
By Lemma \ref{recursiveprops}(g), $F_k \equiv \pm 1 \bmod{5}$.  
By quadratic reciprocity, $5 \equiv d^2 \bmod{F_k}$ for some $d\in\Z$.  
Since $p_k | F_k$ in $\OH$, $d \equiv \pm\sqrt{5} \bmod{p_k}$.  
Without loss of generality, suppose $d \equiv\sqrt{5} \bmod{p_k}$.
Let $\bar{P}$ be the reduction of $P_d$ modulo $F_k$, let $\bar{Q}=2^{2k+1}\bar{P} \in E_d(\F_{F_k})$,
 and choose relatively prime $x,y,z\in\Z$ so that $[x:y:z]\equiv \bar{Q}\bmod F_k$.
Let $Q=2^{2k+1}{P}\in E(H)$, and write $Q=[Q_x:Q_y:Q_z]$ with 
$Q_x,Q_y,Q_z$ relatively prime in $\O_H$.
Identifying $\OH/(p_k)$ with $\F_{F_k}$
and $E(\OH/(p_k))$ with $E_d(\F_{F_k})$, then the reduction of $R$ modulo $p_k$ is $\bar{Q}$.
By Theorem \ref{mainthm}(b) we have
$(Q_z) + (p_k) = \OH$ and $Q_y \equiv 0 \bmod{p_k}$.
Thus, $\bar{Q}$ has order $2$ in $E(\F_{F_k})$, which implies $\gcd(z, F_k) = 1$ and $y \equiv 0 \bmod{F_k}$, as desired.

For the converse, let $d\in\Z$ with $d^2\equiv 5\bmod F_k$, let $x,y,z\in\Z$ be relatively prime with 
$[x:y:z]\equiv 2^{2k+1}P_d\bmod F_k$, and assume $\gcd(z, F_k) = 1$ and $y \equiv 0 \bmod{F_k}$.
Suppose for the sake of contradiction that $F_k$ is composite.
Then we may choose a prime divisor $p$ of $F_k$ such that $p \leq \sqrt{F_k}.$
Since $p | F_k$, the elliptic curve $E_d$ has good reduction modulo $p$, by Lemma \ref{goodreduction}(b). 
It follows from our assumptions that $p \nmid z$ and $y \equiv 0 \bmod{p}$, thus if $\bar{P}$ denotes
the reduction of $P_d$ modulo $p$, then in $E_d(\F_p)$ we have $2^{2k+1}\bar{P}\ne 0$ and $2^{2k+2}\bar{P}=0$.
So the point $\bar P$ has order $2^{2k+2}$ in $E_d(\F_p)$, which gives a lower bound on $|E_d(\F_p)|$.
Applying the Hasse bound, we have
$$
2^{2k+2} \le |E_d(\F_p)| \le (1 + \sqrt{p})^2 \le (1 + F_k^{\frac{1}{4}})^2  \le (1 + 3^{\frac{1}{4}}\cdot2^{\frac{2k+3}{4}})^2
$$
which is a contradiction for all $k > 2$, including all $k\in S$.
This $F_k$ must be prime, and we have proved Theorem \ref{equivdrew}. $\qed$

\subsection{Correctness of Algorithm \ref{sutherlandalg}}\label{algcorrectness}

We now prove that Algorithm \ref{sutherlandalg} produces the correct output for all $k\in S$.

Suppose Algorithm \ref{sutherlandalg} returns ``$F_k$ is prime."  Then
$5$ is a square modulo ${F_k}$,
$y \equiv 0 \bmod{F_k}$, and $\gcd(z,F_k) = 1.$
Thus part (b) of Theorem \ref{equivdrew} is satisfied, and the theorem implies that $F_k$ is prime.
Taking the contrapositive, if $F_k$ is composite then Algorithm \ref{sutherlandalg} outputs ``$F_k$ is composite''.

Now suppose $F_k$ is prime.
By Lemma \ref{recursiveprops}(b,g), $F_k$ is a square modulo $5$, and by quadratic reciprocity,
$5$ is a square modulo $F_k$.
By Euler's criterion, we have $5^{(F_k - 1)/2}\equiv 1 \bmod{F_k}.$
Since $F_k$ is prime, the square roots of $1$ modulo $F_k$ are $\pm 1$.
Thus, $5^{(F_k - 1)/4} \equiv \pm 1 \bmod{F_k}$, so
Algorithm \ref{sutherlandalg} does not terminate at step 1.
 
If $5^{(F_k - 1) /4} \equiv 1 \bmod{F_k}$ then 
$$
d^2 \equiv 5^{2\ell + 2} \equiv 5\cdot 5^{2\ell + 1} \equiv 5\cdot5^{(F_k - 1) / 4} \equiv 5 \bmod{F_k}.
$$
By Lemma \ref{recursiveprops}(h) we have 
$F_k \equiv 5 \bmod{8}$, which implies that $\left(\frac{2}{F_k}\right) = -1$.
Thus if $5^{(F_k - 1)/4} \equiv -1 \bmod{F_k}$, then
$$
d^2 \equiv 2^{4\ell + 2}\cdot 5^{2\ell + 2} \equiv 
2^{(F_k - 1)/2} \cdot 5  \cdot5^{(F_k - 1) / 4} 
\equiv 5 \bmod{F_k},
$$
and this proves that Algorithm \ref{sutherlandalg} does not terminate at step 4.

By Theorem \ref{equivdrew}, Algorithm \ref{sutherlandalg} 
outputs ``$F_k$ is prime'' in step 7.\qed

\subsection{Computations}

Using Algorithm~\ref{sutherlandalg}, we determined all the values of $k\in S$ up to $10^6$ for which $F_k$ is prime.
These computations were performed on a 48-core AMD Opteron system running at 800MHz over the course of several months.

As described in \S 5B of \cite{assw}, we first sieved the set $\mathcal{S} = S\cap [1,10^6]$ to eliminate values of $k$ for which $F_k$ is divisible by small primes; in this case we sieved for prime factors of $F_k$ up to $B=10^{11}$ by applying the recurrence relation in  Proposition~\ref{recursiveFk} modulo each of the primes $p\le B$.
This allowed us to very quickly compute $F_k\bmod p$ for $k\le 10^{6}$, and we then removed from $\mathcal{S}$ all $k$ for which $F_k > p$ and $F_k\equiv 0\bmod p$.
This left approximately $20,000$ values of $F_k$ to which we applied Algorithm~\ref{sutherlandalg}.

In almost every case, composite $F_k$ were identified in step 1 of Algorithm~\ref{sutherlandalg}, which involves just a single exponentiation modulo $F_k$.
We eventually found nine values of $k$ for which $F_k$ is prime, obtaining the following theorem.

\begin{theorem}\label{thm:primeFk}
The values of $k\in S$ with $k\le 10^6$ for which $F_k$ is prime are
\begin{center}
\num{9}, \num{123}, \num{3585}, \num{16253}, \num{17145}, \num{79023}, \num{100619}, \num{501823}, \num{696123}.
\end{center}
\end{theorem}

\end{document}